\documentclass{siamart220329}
\usepackage{amsfonts,mathtools,comment}
\usepackage{color}
\usepackage{enumitem,accents}
\usepackage{subfig}
\usepackage{graphicx}
\usepackage{pinlabel}

\DeclareMathOperator{\mce}{melC}
\DeclareMathOperator{\mcs}{mskC}
\DeclareMathOperator{\cl}{cl}
\DeclareMathOperator{\ccl}{cct}
\DeclareMathOperator{\Comp}{comp}

\newcommand{\expleend}{\hfill\vbox{\hrule height0.6pt\hbox{%
\vrule height1.3ex width0.6pt\hskip0.8ex%
\vrule width0.6pt}\hrule height0.6pt%
}}

\DeclareMathOperator{\act}{act}
\DeclareMathOperator{\inact}{inact}
\newcommand{\inta}{\mathop{\mathrm{int}}}
\newcommand{\exta}{\mathop{\mathrm{ext}}}
\newcommand{\insa}{\mathop{\mathrm{ins}}_{\prec}}
\newcommand{\outa}{\mathop{\mathrm{out}}_{\prec}}

\newcommand{\bG}{\mathbb{G}}

\newcommand{\vect}[1]{\mathbf{#1}}
\newcommand\sbullet{\mathbin{\vcenter{\hbox{\scalebox{0.4}{$\bullet$}}}}}
\newcommand{\dott}[1]{\accentset{\sbullet}{#1}}
\DeclareMathOperator{\ba}{\backslash}
\DeclareMathOperator{\con}{/}
\DeclareMathOperator{\IA}{\mathrm{IA}}
\DeclareMathOperator{\AO}{\mathrm{AO}}
\DeclareMathOperator{\AN}{\mathrm{AN}}

\newsiamremark{example}{Example}
\newsiamremark{remark}{Remark}

\title{An activities expansion of the transition polynomial of a multimatroid}
\author{Criel Merino\thanks{Instituto de Matem\'aticas, Universidad Nacional Aut\'onoma de M\'exico, 04510, M\'exico
(\email{merino@matem.unam.mx}) \funding{Supported by Grant CONACYT A1-S-8195}.}
\and
Iain Moffatt\thanks{Royal Holloway, University of London, Egham, TW20 0EX, UK (\email{iain.moffatt@rhul.ac.uk})
\funding{Supported the Engineering and Physical Sciences Research Council [grant number EP/W033038/1]}.}
\and
Steven Noble\thanks{University of Leeds, Woodhouse, Leeds,
LS2 9JT, UK
(\email{s.d.noble@leeds.ac.uk})\funding{Supported the Engineering and Physical Sciences Research Council [grant number EP/W032945/1].
\\[2mm] For the purpose of open access, the authors have applied a Creative Commons
Attribution (CC BY) licence to any Author Accepted Manuscript version arising. No underlying data is associated with this article.}}}
\headers{Activities expansion of the transition polynomial}
{Criel Merino, Iain Moffatt and Steven Noble}

\begin{document}
\maketitle
\newcommand{\BibTeX}{{\scshape Bib}\TeX\xspace}
\begin{abstract}
 The weighted transition polynomial of a multimatroid is a generalization of the Tutte polynomial. By defining the activity of a skew class with respect to a basis in a multimatroid, we obtain an activities expansion for the weighted transition polynomial. 
 We also decompose the set of all transversals of a multimatroid as a union of subsets of transversals. Each term in the decomposition has the structure of a boolean lattice, and each transversal belongs to a number of terms depending only on the sizes of some of its skew classes. Further expressions for the transition polynomial of a multimatroid are obtained via an equivalence relation on its bases and by extending Kochol's theory of compatible sets.

 We apply our multimatroid results to obtain a result of Morse about the transition polynomial of a delta-matroid and get a partition of the boolean lattice of subsets of elements of a delta-matroid determined by the feasible sets.
 Finally, we describe how multimatroids arise from graphs embedded in surfaces and apply our results to obtain an activities expansion for the topological transition polynomial.
 Our work extends results for the Tutte polynomial of a matroid. 
 \end{abstract}


\section{Introduction}
Tutte's eponymous polynomial is perhaps the most widely studied two variable graph and matroid polynomial due to its many specializations, their
vast breadth and the richness of the underlying theory. 
Its specializations include the chromatic, flow and reliability polynomials from graph theory, the partition functions of the Ising and $q$-states Potts models from theoretical physics, and, via Tait graphs, the Jones polynomial of an alternating knot. Discussion of the Tutte polynomial and closely related polynomials fills an entire handbook~\cite{zbMATH07553843}.

Tutte first introduced the Tutte polynomial of a graph, as the \emph{dichromate} in~\cite{zbMATH03087501}. It is closely related to Whitney's rank
generating function~\cite{MR1503085} which Tutte and Crapo extended from graphs to matroids~\cite{MR262095,Tutte-thesis}. 
See Farr~\cite{Farr-Chapter} for more on the early history of the Tutte polynomial. 

The simplest definition of the Tutte polynomial $T(M;x,y)$ of a matroid $M$ is probably in terms of the rank function $r$ of $M$. (We assume familiarity with the basic notions of matroids. See~\cite{MR1207587}
for more information.)
\begin{definition}
For a matroid $M=(E,r)$, we let
\[ T(M;x,y) :=   \sum_{A\subseteq E} (x-1)^{r(E)-r(A)} (y-1)^{|A|-r(A)}.\]
\end{definition}
Tutte's original definition~\cite{zbMATH03087501} used the notion of activities.
Take a matroid $M$ with element set $E$ and fix a total ordering $\prec$ of $E$. Given a basis $B$ of $M$, for each element $e$ of $E-B$, there is a unique circuit, $C(B,e)$ such that $C(B,e)\subseteq
B\cup \{e\}$. 
The circuit $C(B,e)$ is called the \emph{fundamental circuit of $e$ with respect to $B$.}
An element $e$ of $E-B$ is \emph{externally active} with respect to $B$ if $e$ is the least element of $C(B,e)$ with respect to $\prec$.
Similarly, for each element $e$ of $B$, there is a unique cocircuit, $C^*(B,e)$ such that $C^*(B,e)\subseteq (E-B)\cup \{e\}$. 
The cocircuit $C^*(B,e)$ is called the \emph{fundamental cocircuit of $e$ with respect to $B$.}
An element $e$ of $B$ is
\emph{internally active} with respect to $B$ if $e$ is the least element of $C^*(B,e)$ with respect to $\prec$.
The collections of internally and externally active elements with respect to $B$ are denoted by $\inta_{M,\prec}(B)$ and $\exta_{M,\prec}(B)$, respectively. We usually omit the subscript $M$ whenever it is clear from context. Given that,
in general, $|{\inta_{\prec}(B)}|$ and $|{\exta_{\prec}(B)}|$ depend on $\prec$, the following result proved by Tutte~\cite{zbMATH03087501} for graphs and extended by Crapo~\cite{MR262095} to matroids is remarkable. Let $\mathcal B(M)$ denote the collection of bases of a matroid $M$.
\begin{theorem}[Tutte~\cite{zbMATH03087501}, Crapo~\cite{MR262095}]\label{thm:tutteact}
Let $M$ be a matroid
and let $\prec$ be a total ordering of its elements. Then
\[ T(M;x,y) = \sum_{B\in \mathcal B(M)} x^{|{\inta_{\prec}(B)}|}y^{|{\exta_{\prec}(B)}|}.\]
\end{theorem}
Setting $x=y=2$, we get the following.
\begin{corollary}\label{cor:x=y=2}
Let $M$ be a matroid
and let $\prec$ be a total ordering of its elements. Then
\[  \sum_{B\in \mathcal B(M)} 2^{|{\inta_{\prec}(B)}|+|\exta_{\prec}(B)|}=2^{|E|}.\]
\end{corollary}
Crapo gave a nice bijective proof of this identity. For any basis $B$ of $M$ consider the Boolean interval $I(B):=[B-\inta_{\prec}(B), B\cup \exta_{\prec}(B)]$, that is,
the family of sets $X$ with $B-\inta_{\prec}(B) \subseteq X \subseteq B\cup \exta_{\prec}(B)$. Then we have the following theorem of Crapo from~\cite{MR262095}.

\begin{theorem}[Crapo~\cite{MR262095}]\label{thm:partition}
Let $M$ be a matroid with element set $E$
and let $\prec$ be a total ordering of $E$. Then the
collection of Boolean intervals $\{ I(B): B \in \mathcal B\}$ forms a partition of $2^{E}$.
\end{theorem}

Dawson~\cite{zbMATH03821756} showed that the existence of such a partition into Boolean intervals does not actually require matroid structure. In fact such a partition exists for any family of sets, although clearly in this more general setting the matroidal description of the Boolean intervals is lost. One of the main aims of this paper is to show that an analogue of Theorem~\ref{thm:partition} reflecting the underlying structure holds in the setting of multimatroids.

\smallskip

Our interest in multimatroids  comes  from topological graph theory (i.e., the study of graphs embedded in surfaces) via delta-matroids. 
Delta-matroids were  introduced independently by several authors~\cite{MR904585,zbMATH04070920,zbMATH03985246} but principally studied by 
Bouchet who showed that every embedded graph has an associated delta-matroid in the same way as every graph has an associated matroid~\cite{zbMATH04185622,CMNR1}. We give the formal definition of a delta-matroid later in Section~\ref{sec:dm}, but roughly speaking delta-matroids generalize
matroids by allowing the bases to have differing sizes.
Some important classes of delta-matroids mirror classes of matroids, for example, those coming from embedded graphs and those coming from binary matrices. But, in contrast with graphic matroids, the class of delta-matroids coming from embedded graphs can also be realized as a class defined in terms of Eulerian circuits in $4$-regular graphs~\cite{zbMATH04185622}. As the delta-matroid of an embedded graph with one vertex may be obtained from its intersection graph and these graphs are circle graphs, there is also a class of delta-matroids associated with circle graphs and these form a subclass of those coming from embedded graphs. For more information on these classes see~\cite{zbMATH07307106,me}.

 In topological graph theory nonorientable surfaces present additional challenges compared with orientable surfaces. 
In~\cite{CMNR2} we studied how nonorientability  
impacts the delta-matroid of an embedded graph but here we show that to fully understand the basis structure of a graph embedded in a nonorientable surface we need to consider a more general structure, namely, a multimatroid (see Section~\ref{sec:multim} for a definition). These were introduced by Bouchet in the 1990s
~\cite{MM1zbMATH01116184,MM2zbMATH01119073,MM3zbMATH01648954}, extending and unifying delta-matroids and 
isotropic systems~\cite{zbMATH04045764},
another of his earlier combinatorial structures.    (Very roughly, isotropic systems generalize circuit spaces of binary matroids.) 

Special classes of multimatroids include
$1$-matroids which correspond precisely to matroids and $2$-matroids which are equivalent (in a sense we will make precise in~Section~\ref{sec:dm}) to delta-matroids.
Thus, because every matroid is a delta-matroid, matroids are $1$-matroids and can also be seen as $2$-matroids. A subclass of $3$-matroids is equivalent to isotropic systems but of more relevance to us is that $3$-matroids of embedded graphs allow us, in some sense, to capture when edges of graphs embedded in a nonorientable surface go over crosscaps.

\smallskip

In this paper we are interested in basis activities in multimatroids. 
Our first two main results extend Theorems~\ref{thm:tutteact} and~\ref{thm:partition} to multimatroids. It is, however, more natural to describe our results in terms of the transition polynomial~\cite{MR3191496} rather than the Tutte polynomial.
For classes of objects where these are both defined, they are related by a substitution of variables. Moreover, the transition polynomial of an embedded graph includes Bollob\'as and Riordan's 2-variable ribbon graph polynomial~\cite{MR1851080,MR1906909} as a specialization.
Theorem~\ref{thm:main1} extends Crapo and Tutte's Theorem~\ref{thm:tutteact} by expressing the transition polynomial of a multimatroid as a sum over its bases.
This expansion of the transition polynomial launches us into the study of bases with our second main result, Theorem~\ref{thm:maingen2}, extending Crapo's partition into  Boolean intervals in Theorem~\ref{thm:partition} to multimatroids.   

Recent results of Kochol~\cite{zbMATH07725327} suggest an alternative approach to obtain an expression for the transition polynomial. In Section~\ref{s:kochol}, we generalize the main ideas of Kochol by defining a notion of compatibility for multimatroids and obtaining a corresponding expression for the transition polynomial in Theorem~\ref{thm:cocomp}, our third main result. 

The notion of compatibility for multimatroids allows us to define an equivalence relation $\sim$ on the transversals of a multimatroid. 
In the second half of Section~\ref{s:kochol} we relate the equivalence classes of $\sim$ to the sets appearing in Theorem~\ref{thm:maingen2} which leads to Theorem~\ref{thm:equiv}, our fourth main result, which in turn leads to another expression for the transition polynomial.

In the remainder of the paper we exploit Theorems~\ref{thm:main1} and~\ref{thm:maingen2}. In Section~\ref{sec:dm} we use the equivalence of $2$-matroids and delta-matroids to derive Theorem~\ref{thm:delta1} from Theorem~\ref{thm:main1}, recovering the main result of Morse in~\cite{zbMATH07067836}. We then apply Theorem~\ref{thm:maingen2} to construct a partition of the subsets of the elements of a delta-matroid into Boolean intervals in Theorem~\ref{thm:delta2}, making concrete implicit underlying ideas in~\cite{zbMATH07067836}. 

Finally, we return to our original motivation by collecting ideas from the literature to show that there is a $3$-matroid associated with every embedded graph, giving Theorem~\ref{newtt1}. We use this to prove 
 an expression for the transition polynomial of an embedded graph~\cite{MR2869185} which follows by combining Theorems~\ref{thm:main1} and~\ref{thm:newrg3} with a careful interpretation of how active edges are manifest in an embedded graph.

\section{Multimatroids}\label{sec:multim}
Multimatroids generalize matroids.  
They were introduced by Bouchet in~\cite{MM1zbMATH01116184}. They are defined on what he called a \emph{carrier}, a pair
$(U,\Omega)$ where $U$ is a finite set and $\Omega$ is a partition of $U$ into (non-empty) blocks called \emph{skew classes}. If each skew class
has size $q$, then $(U,\Omega)$ is a \emph{$q$-carrier}.
A multimatroid is said to be \emph{non-degenerate} if each of its skew classes has cardinality at least two. 
The first sentence in each of Examples~\ref{examplenodot} and~\ref{example} offers examples of carriers. 

In order to prepare the reader for the definitions which follow, we begin by giving a hint of how multimatroids generalize matroids.
We shall see that $q$-carriers give rise to a class of multimatroids called $q$-matroids. 1-matroids are exactly matroids. But there is also a correspondence between matroids and a class of 2-matroids that turns out to be more useful. 
In Example~\ref{eg:matroid}, we describe how to form a $2$-matroid from a matroid. In this example, every element of the matroid gives rise to a skew class in the multimatroid. Thus there is a bijection between the elements of the matroid and the skew classes of the corresponding multimatroid. More generally, 
we shall see that given a multimatroid defined on a carrier $(U,\Omega)$, it is helpful to think of the skew classes, rather than the elements of $U$, as playing a role analogous to the elements of a matroid, and the elements of each skew class representing different `states' that the skew class might play.

A pair of distinct elements belonging to the same skew class is called a \emph{skew pair}. A \emph{subtransversal} of $\Omega$ is a subset of $U$
meeting each skew class at most once. A \emph{transversal} $T$ of $\Omega$ is a subtransversal of $\Omega$ with $|T|=|\Omega|$, that is, a set containing precisely one element from each skew class. Given a subtransversal $S$ and a skew class $\omega$ meeting $S$, we let $S_{\omega}$ denote the unique element of $S\cap \omega$.
The sets of
subtransversals and transversals of $\Omega$ are denoted by $\mathcal S(\Omega)$ and $\mathcal T(\Omega)$ respectively.

A \emph{multimatroid} is a triple $Z=(U,\Omega,r_Z)$ so that $(U,\Omega)$ is a carrier and $r_Z$ is a \emph{rank function},
a non-negative integral function defined on $S(\Omega)$
and satisfying the two axioms below. 
\begin{enumerate}[label = {(R\arabic*)}]
\item\label{axiom1} For each transversal $T$ of $\Omega$, the pair comprising $T$ and the restriction of $r_Z$ to subsets of $T$ forms a matroid.
\item\label{axiom2} For each subtransversal $S$ of $\Omega$ and skew pair $\{x,y\}$ from  a skew class disjoint from $S$, we have
    \[ r_Z(S\cup \{x\}) + r_Z(S\cup \{y\}) - 2r_Z(S) \geq 1.\]
\end{enumerate}
We generally omit the subscript $Z$ from $r_Z$ whenever the context is clear.  Axiom~\ref{axiom1} implies the following: $r(\emptyset)=0$; for $A\in\mathcal S(\Omega)$ and  $x\in U$ such that $A\cup\{x\}\in\mathcal S(\Omega)$, $r(A)\leq r(A\cup\{x\})\leq r(A)+1$; and for $A$, $B\in\mathcal S(\Omega)$ such that $A\cup B\in\mathcal S(\Omega)$, $r(A\cup B)+r(A\cap B)\leq r(A)+r(B)$.

 It is convenient to extend the concepts we have defined for
carriers to multimatroids. For example, by a transversal of a multimatroid we mean a transversal of its underlying carrier,
and we sometimes use $\mathcal T(Z)$ to denote the set of transversals of the carrier of $Z$.
We say that a multimatroid is a \emph{$q$-matroid} if its carrier is a $q$-carrier.
The special case of $2$-matroids was first introduced earlier than general multimatroids in~\cite{MR904585} under the name
\emph{symmetric matroids}. The \emph{order} of a carrier $(U,\Omega)$, denoted by $|\Omega|$,
is its number of skew classes. If $S$ is a subtransversal of $Z$, then we define its \emph{nullity} $n_Z(S)$ by $n_Z(S)=|S|-r_Z(S)$, again omitting
the subscript when the context is clear.

An \emph{independent set} of $Z$ is a subtransversal $S$ with $n(S)=0$. Subtransversals that are not independent are \emph{dependent}. A
\emph{basis} is a maximal independent subtransversal.
A \emph{circuit} is a minimal dependent subtransversal. Given a carrier $(U,\Omega)$, a multimatroid on $(U,\Omega)$ is determined by each of its collections of independent sets, circuits and bases~\cite{MM1zbMATH01116184}. The set of bases and circuits of $Z$ are denoted by $\mathcal B(Z)$ and $\mathcal C(Z)$, respectively.
Observe that for any subtransversal $S$, we have $r(S)=\max_{B\in\mathcal B(Z)} |B\cap S|$.
The following are straightforward consequences of~\ref{axiom1} and~\ref{axiom2}.
\begin{proposition}[{\cite[Proposition~5.5]{MM1zbMATH01116184}}]\label{prop:basestrans}
The bases of a non-degenerate multimatroid $Z$ are transversals of $Z$.
\end{proposition}

\begin{proposition}[{\cite[Proposition~5.7]{MM1zbMATH01116184}}]\label{prop:onlyonecircuit}
Let $B$ be a basis and $\omega$ be a skew class of a multimatroid. Then $B\cup \omega$ contains at most one circuit.
\end{proposition}

A circuit of the form described in the previous proposition is a \emph{fundamental circuit}~\cite{MM3zbMATH01648954} with respect to $B$ and is denoted by $C(B,\omega)$.

Let $Z=(U,\Omega,r)$ and let $u$ be in $U$. Then the \emph{elementary minor}~\cite{MM2zbMATH01119073} of $Z$ by $u$, denoted by $Z |u$, is the triple
$(U',\Omega',r')$ with $\Omega'=\{\omega \in \Omega: u\notin \omega\}$, $U'=\bigcup_{\omega\in \Omega'}\omega$
and $r'(S)=r(S\cup\{u\})-r(\{u\})$, $S\in\mathcal S(\Omega')$.
It is straightforward to show that $Z|u$ is a multimatroid.

Let $Z=(U,\Omega,r)$ be a multimatroid and $\omega$ be a skew class of $\Omega$.
We let $Z \ba \omega$ denote the \emph{restriction} of $Z$ to $U-\omega$~\cite{MM1zbMATH01116184}, comprising the triple
$(U-\omega,\Omega-\{\omega\},r_{Z\ba \omega})$,
 where
$r_{Z\ba \omega}$ is the restriction of $r$ to subtransversals of $\Omega- \{\omega\}$.
It is straightforward to show that $Z\ba \omega$ is a multimatroid.
It follows immediately that if $S$ is a subtransversal of $\Omega- \{\omega\}$, then
$n_{Z\ba \omega}(S)=n_Z(S)$.

An element $e$ of a multimatroid $(U,\Omega,r)$ is \emph{singular} if $r(\{e\})=0$. A skew class is \emph{singular} if it contains a singular element.
We establish a few simple properties of singular elements and skew classes.
\begin{lemma}\label{lem:rankssingular}
Let $e$ be a singular element of a multimatroid $(U,\Omega,r)$. 
Suppose that the skew class containing $e$ is $\omega$ and that $S$ is a subtranversal of $\Omega$ such that $S \cap \omega=\emptyset$. Then $r(S\cup \{e\})=r(S)$ and $r(S\cup \{u\})=r(S)+1$, for all $u\in \omega-\{e\}$.
\end{lemma}

\begin{proof}
By \ref{axiom1} we have $r(S)\leq r(S\cup \{e\}) \leq r(S)+r(\{e\}) = r(S)$, giving $r(S\cup \{e\})=r(S)$. Now, by \ref{axiom2}, $r(S\cup \{u\})=r(S)+1$.
\end{proof}

The following results of Bouchet follow immediately.
\begin{proposition}[{\cite[Proposition~5.4]{MM2zbMATH01119073}}]
A singular skew class contains precisely one singular element.
\end{proposition}

\begin{proposition}[{\cite[Proposition~5.5]{MM2zbMATH01119073}}]
\label{prop:singminorsanddelclass}
If $\omega$ is a singular skew class of a multimatroid $Z$, then for all $x$ in $\omega$, we have $Z|x = Z\ba \omega$.
\end{proposition}

Bouchet established several cryptomorphisms for multimatroids and in particular proved the following.
\begin{proposition}
[{\cite[Proposition~5.4]{MM1zbMATH01116184}}]
\label{prop:uniqueskewpair}
Let $Z$ be a multimatroid and $C_1$ and $C_2$ be circuits of $Z$. Then $C_1 \cup C_2$ cannot contain  exactly one skew pair.
\end{proposition}

From this we get an easy proof of the next lemma describing the circuits of a multimatroid with a singular skew class.

\begin{lemma}\label{lem:circuits}
Let $Z=(U,\Omega,r)$ be a multimatroid with a singular element $e$. Suppose that $\omega$ is the skew class containing $e$. Then $\mathcal C(Z) = \mathcal C(Z\ba \omega) \cup \{\{e\}\}$.
\end{lemma}

\begin{proof}
First, observe that as circuits are minimal dependent sets, the only circuit of $Z$ containing $e$ is $\{e\}$. Together with Proposition~\ref{prop:uniqueskewpair} this implies that the only circuit of $Z$ meeting $\omega$ is $\{e\}$.
Second, recall that for every transversal $T$ of $\Omega -\{\omega\}$, we have $r_Z(T)=r_{Z\ba \omega}(T)$. Thus $C$ is a circuit of $Z\ba \omega$ if and only if it is a circuit of $Z$ with $C\cap \omega=\emptyset$. The result follows by combining these two observations.
\end{proof}

We record the following  lemma for reference later.
\begin{lemma}\label{lem:usefulnull}
Let $Z$ be a multimatroid with a non-singular skew class $\omega$. Let $S$ be a subtransversal of $Z$ with $S\cap \omega =\emptyset$ and let $u$ be in $\omega$. Then $n_Z(S\cup \{u\})=n_{Z|u}(S)$.
\end{lemma}
\begin{proof}
As $\omega$ is non-singular and $u\in \omega$, $r_Z(\{u\})=1$. Hence
\begin{align*} n_{Z|u}(S) &= |S|-r_{Z|u}(S) = |S|-(r_Z(S\cup\{u\}) -r_Z(\{u\}))\\
&= |S\cup \{u\}| - r_Z(S\cup\{u\}) = n_Z(S\cup \{u\}),\end{align*}
as required.
\end{proof}

\begin{example}\label{examplenodot}
 Let $(U,\{ \omega_1,\omega_2,\omega_3\})$ 
 be the carrier with
 $\omega_1=\{d, e, f\}$,
 $\omega_2=\{g, h\}$ and
 $\omega_3=\{i, j\}$. An example of a multimatroid  $Z$ over this carrier has the following 8 bases:
\begin{center}
\begin{tabular}{cccc}
$\{ d, g, j\}$,&
$\{ d, h, i\}$,&
$\{ d, h, j\}$,&
$\{ e, g, i\}$,\\
$\{ e, g, j\}$,&
$\{ e, h, j\}$,&
$\{ f, g, i\}$,&
$\{ f, h, i\}$.
\end{tabular}
\end{center}

The remaining four transversals have rank two.
The circuits are the following three subtransversals:
$\{ d, g, i\}$,
$\{ e, h, i\}$ and
$\{ f, j\}$.
\expleend\end{example}

Notice that the elements within a skew class of a multimatroid are not ordered. Using distinct letters for each element of a multimatroid can make it difficult for the reader to quickly appreciate the structure of an example. To overcome this, in the following example, we denote the elements of a skew class by decorating the same letter in different ways. This has the disadvantage of suggesting that elements of different skew classes with the same decoration are somehow related. We stress that this is not the case in general, but in Section~\ref{sec:newtopo} we describe the $3$-matroid of a ribbon graph where such a relation does arise as the elements may be naturally partitioned into three transversals. 

\begin{example}\label{example}
 Let $(U,\{ \omega_a,\omega_b,\omega_c\})$
 be the 3-carrier with
 $\omega_a=\{\dott a, \overline{a}, \widehat{a}\}$,
 $\omega_b=\{\dott b, \overline{b}, \widehat{b}\}$ and
 $\omega_c=\{\dott c, \overline{c}, \widehat{c}\}$. An example of a 3-matroid  $Z$ over this carrier has the following 16 bases:
\begin{center}
\begin{tabular}{cccc}
$\{ \dott {a}, \overline{b}, \overline{c}\}$,&
$\{ \widehat{a}, \dott {b}, \overline{c}\}$,&
$\{ \dott {a}, \widehat{b}, \dott {c}\}$,&
$\{ \widehat{a}, \widehat{b}, \dott {c}\}$,\\
$\{ \overline{a}, \dott {b}, \overline{c}\}$,&
$\{ \widehat{a}, \dott {b}, \dott {c}\}$,&
$\{ \dott {a}, \dott {b}, \widehat{c}\}$,&
$\{ \widehat{a}, \widehat{b}, \widehat{c}\}$,\\
$\{ \dott {a}, \dott {b}, \dott {c}\}$,&
$\{ \dott {a}, \widehat{b}, \overline{c}\}$,&
$\{ \dott {a}, \overline{b}, \widehat{c}\}$,&
$\{ \widehat{a}, \overline{b}, \widehat{c}\}$,\\
$\{ \widehat{a}, \overline{b}, \overline{c}\}$,&
$\{ \overline{a}, \widehat{b}, \overline{c}\}$,&
$\{ \overline{a}, \dott {b}, \widehat{c}\}$,&
$\{ \overline{a}, \widehat{b}, \widehat{c}\}$.
\end{tabular}
\end{center}

Of the remaining eleven transversals, all of which are dependent,
$\{ \overline{a}, \overline{b}, \dott {c}\}$
 has rank one and the rest have rank two.
The circuits are the following seven subtransversals:
$\{ \overline{a}, \overline{b}\}$,
$\{ \overline{a}, \dott {c}\}$,
$\{ \overline{b}, \dott {c}\}$,
$\{ \dott {a}, \dott {b}, \overline{c}\}$,
$\{ \widehat{a}, \widehat{b}, \overline{c}\}$,
$\{ \widehat{a}, \dott {b}, \widehat{c}\}$ and
$\{ \dott {a}, \widehat{b}, \widehat{c}\}$.
We will use $Z$ as a running example through much of the paper.
\expleend\end{example}

 We have noted that every matroid is a $1$-matroid, but there is also a simple construction which associates a $2$-matroid with every matroid. (Although not every 2-matroid arises in this way.) This will be more useful as non-degeneracy will be a requirement of many of our results.
\begin{example}\label{eg:matroid}
    Let $M$ be a matroid with element set $E$ and rank function $r_M$. We define a $2$-matroid $Z(M):=(U,\Omega,r_{Z(M)})$ as follows. We have $U:=\{\dott e,\overline e:e\in E\}$ and $\Omega:=\{\omega_e:e\in E\}$, where for each $e$ we have $\omega_e:=\{\dott e,\overline e\}$.
    For a subset $A$ of $E$, we define $\dott A:=\{\dott e:e\in A\}$ and $\overline A:=\{\overline e:e\in E\}$. 
    Then for disjoint subsets $A$ and $B$ of $E$, we define $r_{Z(M)}(\dott A\cup \overline B)= r_M(A) + r_{M^*}(B)$, where $M^*$ is the dual matroid of $M$. It is not difficult to show that that $r_{Z(M)}$ satisfies \ref{axiom1}. To see that it satisfies \ref{axiom2}, let $S$ be a subtransversal of $\Omega$ missing the skew class $\omega_e$. Then $S=\dott A \cup \overline B$ for disjoint subsets of $A$ and $B$ of $E$. Proposition~2.1.11 of \cite{MR1207587} states that a circuit $C$ and a cocircuit $C^*$ of a matroid satisfy $|C\cap C^*|\ne 1$. Thus, in $M$, either $A\cup\{e\}$ does not have a circuit containing $e$ or $B\cup\{e\}$ does not have a cocircuit containing $e$. Hence either $r_M(A\cup\{e\})=r_M(A)+1$ or $r_{M*}(B\cup \{e\})=r_{M^*}(B)+1$, giving either  
$r_{Z(M)}(\dott A\cup \overline B\cup\{\dott e\})
= r_{Z(M)}(\dott A\cup \overline B) +1$ or
$r_{Z(M)}(\dott A\cup \overline B\cup\{\overline e\})
= r_{Z(M)}(\dott A\cup \overline B) +1$.

We let $\phi:2^E\rightarrow \mathcal T(\Omega)$ be given by $\phi(A)=\dott A \cup \overline {E-A}$. Then clearly $\phi$ is a bijection. We shall see several examples of classical matroid constructions or results which may be translated from a matroid $M$ to the $2$-matroid $Z(M)$ via $\phi$. For example,
\[ \mathcal {B}(Z(M)) = \{ \phi(B) : B\in \mathcal B(M)\},\] which implies that $\phi|_{\mathcal B(M)}$ is a bijection from $\mathcal B(M)$ to $\mathcal B(Z(M))$.

Next we claim that $C$ is a circuit of $Z(M)$ if and only if $C=\dott{C'}$ for a circuit $C'$ of $M$ or $C=\overline{C'}$ for a cocircuit $C'$ of $M$. It is not difficult to see that this claim is true when $C\subseteq \dott E$ or $C\subseteq \overline E$, so it remains to show that $Z(M)$ has no circuit meeting both $\dott E$ and $\overline E$. Suppose, for contradiction, that $C$ is a circuit of $Z(M)$ meeting both $\dott{E}$ and $\overline E$. Thus there are non-empty, disjoint subsets $A$ and $B$ of $E$ with $C=\dott A\cup \overline B$. 
    Then, as circuits are minimal dependent sets, both $\dott A$ and $\overline B$ are independent in $Z(M)$. So
    \[ r_{Z(M)}(C) = r_M(A) + r_{M^*}(B)= r_{Z(M)}(\dott A) + r_{Z(M)}(\overline B) = |A|+|B|=|C|,\]
    contradicting the fact that $C$ is a circuit.
In Section~\ref{sec:dm} we will see that this construction is a special case of the construction that associates a $2$-matroid with a delta-matroid. We will return to this example as a second running example for much of the paper.
\expleend\end{example}

\section{The transition polynomial}\label{sec:results}
This section contains our first two main results, Theorems~\ref{thm:main1} and~\ref{thm:maingen2}. We begin, in Subsection~\ref{ss:results1}, by defining the weighted transition polynomial and describing some of its basic properties. We then establish an expression for the weighted transition polynomial as a sum over bases, generalising Theorem~\ref{thm:tutteact} from the introduction. Our second main result, appearing in Subsection~\ref{ss:results2}, provides a combinatorial interpretation for a special case, extending Theorem~\ref{thm:partition} from the introduction.

\subsection{An activities expansion}\label{ss:results1}
The \emph{weighted transition polynomial} of a multimatroid $Z=(U,\Omega,r)$, introduced in~\cite{MR3191496}, is
\[ Q(Z;\vect x, t):= \sum_{T\in \mathcal T(\Omega)} t^{n(T)}\vect x_T ,\]
where $\vect x$ is a vector indexed by $U$ with an entry $\vect x_u$ for each $u$ in $U$ and $\vect x_T=\prod_{u\in T} \vect x_u$.
Notice that if $U=\emptyset$, then $\mathcal T(\Omega)=\{\emptyset\}$, so $Q(Z;\vect x,t)=1$.
The \emph{(unweighted) transition polynomial} of $Z$, which we denote by $Q(Z;t)$, is obtained from $Q(Z;\vect x,t)$ by setting $\vect x_u=1$ for all $u$.
 An example of the unweighted transition polynomial can be found at the end of this section.

It follows from Proposition~\ref{prop:basestrans} that whenever $Z$ is non-degenerate, $Q(Z;0)$ is equal to the number of bases of $Z$. 

The next result is a minor rephrasing of Theorem~6 from~\cite{MR3191496}. 
Its equivalence to Theorem~6 from~\cite{MR3191496}
follows from Proposition~\ref{prop:singminorsanddelclass}.
\begin{proposition}[{\cite[Theorem 6]{MR3191496}}]\label{prop:delcon}
Let $Z=(U,\Omega,r)$ be a multimatroid and let $\omega$ be a skew class of $\Omega$. If $\omega$ contains a singular element $e$, then
\[
Q(Z;\vect x,t) = \Big(t\vect x_e+\sum_{u\in\omega-\{e\}} \vect x_u\Big) Q(Z\ba \omega;\vect x',t).\]
If $\omega$ is non-singular, then
\[
Q(Z;\vect x,t) =
\sum_{u\in \omega}\vect x_u Q(Z|u;\vect x', t).
\]
In either case, $\vect x'$ is obtained from $\vect x$ by removing the entries indexed by the elements of $\omega$.
\end{proposition}

Let $Z$ be a multimatroid. Fix an arbitrary total ordering $\prec$ of the skew classes of $Z$. The ordering $\prec$ also induces total orderings of
the skew classes of elementary minors of $Z$, of $Z\ba \omega$ for any $\omega$, and of any subtransversal of $\Omega$,
which are by a slight abuse of notation also
denoted by $\prec$.

Given a subtransversal $S$, we let $\min(S)$ denote the least element of $S$ according to $\prec$. We say that a skew class $\omega$ of $Z$ is \emph{active} with respect to $S$ if there is a circuit $C$ of $Z$ with $\min (C)\in \omega$ and $C-\omega \subseteq S$. Notice that $C$ may or may not be a subset of $S$. 
If $\omega$ is not active with respect to $S$, we say it is \emph{inactive} with respect to $S$.
We let $\act_{Z,\prec}(S)$ and $\inact_{Z,\prec}(S)$ denote, respectively, the sets of active and inactive skew classes with respect to $S$. As usual, we omit $Z$ when the context is clear.

The definitions of active  and inactive are normally only applied to bases, and we will not apply them more widely until Section~\ref{s:kochol}. When $S$ is a basis and $\omega$ is active with respect to $S$, the circuit $C$ in the definition of active is the fundamental circuit $C(S,\omega)$.
 An example illustrating these definitions can be found in Example~\ref{eg:examplectd.1} appearing at the end of this subsection.

Given a basis $B$ and a skew class $\omega$ such that $C(B,\omega)$ exists, we let $\underline B_{\omega}$ denote the unique element of $C(B,\omega)-B$.

\begin{lemma}\label{lem:activitiessingular}
Let $Z$ be a multimatroid with a singular element $e$. Fix an arbitrary total ordering $\prec$ of the skew classes of $Z$.
Let $\gamma$ be the skew class of $Z$ containing $e$ and suppose that $\{e,u\}$ is a skew pair.
Then $B$ is a basis of $Z\ba \gamma$ if and only if $B\cup \{u\}$ is a basis of $Z$. Moreover, if $B$ is a basis of $Z\ba \gamma$, then each of the following holds. 
\begin{enumerate}
    \item\label{item:l321} $\act_{Z,\prec}(B\cup \{u\})=\act_{Z\ba \gamma,\prec}(B) \cup\{\gamma\}$.
    \item\label{item:l322} $\inact_{Z,\prec}(B\cup \{u\})=\inact_{Z\ba \gamma,\prec}(B)$.
\item\label{item:l323} For every skew class $\omega$ of $Z\ba \gamma$, such that $C(B,\omega)$ exists, we have 
    $\underline{B\cup\{u\}}_{\omega}=\underline B_{\omega}$.
\end{enumerate}
\end{lemma}

\begin{proof}
The fact that $B$ is a basis of $Z\ba \gamma$ if and only if $B\cup \{u\}$ is a basis of $Z$ follows from Lemma~\ref{lem:rankssingular}. Items~\ref{item:l321}--\ref{item:l323}. follow from Lemma~\ref{lem:circuits} and the observation that $\gamma$ is active with respect to $B\cup\{u\}$ because it is singular.
\end{proof}

\begin{lemma}\label{lem:activitiesnonsingularprelim}
Let $Z$ be a multimatroid. 
Suppose that $u$ is a non-singular element of $Z$ and $\omega$ is a skew class of $Z|u$.
Then $B$ is a basis of $Z|u$ if and only if $B\cup \{u\}$ is a basis of $Z$. Moreover if $B$ is a basis of $Z|u$, 
then $C_Z(B\cup \{u\},\omega)$ exists if and only if 
$C_{Z|u}(B,\omega)$ exists, and when these fundamental circuits exist we have $C_{Z|u}(B,\omega) = C_Z(B\cup \{u\},\omega)-\{u\}$. 
\end{lemma}

\begin{proof}
The first assertion follows immediately from the definition of the rank function of $Z|u$.
Suppose then that $B$ is a basis of $Z|u$.
Then for every $f$ in $\omega-B$, by the definition of the rank function of $Z|u$, $(B-\omega)\cup\{f\}$ is independent in $Z|u$ if and only if $(B-\omega)\cup\{f,u\}$ is independent in $Z$. Thus $C_Z(B\cup \{u\},\omega)$ exists if and only if 
$C_{Z|u}(B,\omega)$ exists. Suppose then that 
$C_Z(B\cup \{u\},\omega)$ and $C_{Z|u}(B,\omega)$ exist. 
Let $C:=C_Z(B\cup \{u\},\omega)$.
We next show that $C-\{u\}$ is a circuit of $Z|u$. 
 Suppose first that $u\notin C$. As $C$ is the unique circuit in $B\cup\{u\}\cup \omega$, for any subset $C'$ of $C$, we have $r_Z(C'\cup \{u\})=r_Z(C')+1$. Thus $r_{Z|u}(C')=r_Z(C')$, so $C-\{u\}=C$ is a circuit of $Z|u$. 
Now suppose that $u\in C$. Then $r_{Z|u}(C-\{u\})=r_{Z}(C)-1=|C-\{u\}|-1$,
and if $C'$ is a proper subset of $C-\{u\}$, then $r_{Z|u}(C')=r_Z(C'\cup\{u\})-1=|C'|$. Thus $C-\{u\}$ is a circuit in $Z|u$. But $C-\{u\} \subseteq B \cup \omega$, and $B\cup \omega$ contains a unique circuit of $Z|u$, namely $C_{Z|u}(B,\omega)$. Hence $C_{Z|u}(B,\omega)=C-\{u\}$, as required.
\end{proof}

\begin{lemma}\label{lem:activitiesnonsingular}
Let $Z$ be a multimatroid. Fix an arbitrary total ordering $\prec$ on the skew classes of $Z$. Suppose that $\gamma$ is the greatest skew class according to $\prec$, that $\gamma$ is non-singular and $u\in \gamma$. Then $B$ is a basis of $Z|u$ if and only if $B\cup \{u\}$ is a basis of $Z$. Moreover if $B$ is a basis of $Z|u$, then each of the following holds.
\begin{enumerate}
    \item $\act_{Z,\prec}(B\cup \{u\})=\act_{Z|u,\prec}(B)$.
    \item $\inact_{Z,\prec}(B\cup \{u\})=\inact_{Z|u,\prec}(B)\cup \{\gamma\}$.
    \item For every skew class $\omega$ of $Z|u$, such that $C(B,\omega)$ exists, we have 
    $\underline{B\cup\{u\}}_{\omega}=\underline B_{\omega}$.
\end{enumerate}
\end{lemma}

\begin{proof}
After noting that $\gamma$ is inactive with respect to every basis, because it is non-singular and the greatest skew class with respect to $\prec$, each assertion in the lemma follows immediately from the previous lemma.
\end{proof}

Our first main result is the following analogue of Theorem~\ref{thm:tutteact}. The assumption that $Z$ is non-degenerate is required both to ensure that $|\omega|-1\ne 0$ and that the element $B_{\omega}$ is always defined. If, for example, $\omega=\{e\}$ and $r(\{e\})=0$, then no basis meets $\omega$.

\begin{theorem}\label{thm:main1}
Let $Z$ be a non-degenerate multimatroid and let $\prec$ be a total ordering of its skew classes. Then
\[ Q(Z;\vect x,t)= \sum_{B\in \mathcal B(Z)} \Bigg(\prod_{\omega \in \inact_{\prec}(B)} \vect x_{B_{\omega}}\Bigg) \Bigg(\prod_{\omega \in \act_{\prec}(B)}
 \Big(\frac {t\vect x_{\underline B_{\omega}}} {|\omega|-1}+\vect  x_{B_{\omega}}\Big)\Bigg).\]
\end{theorem}

\begin{proof}
Suppose that $Z=(U,\Omega,r)$. We proceed by induction on $|\Omega|$. If $|\Omega|=0$, then $\mathcal B(Z)=\{\emptyset\}$ and $\inact_{\prec}(\emptyset)=\act_{\prec}(\emptyset)=\emptyset$, so $Q(Z;\vect x,t)=1$, as required.
So we may suppose that $|\Omega|>0$. Let $\gamma$ be the greatest skew class according to $\prec$. Suppose first that $\gamma$ is singular and the element $e$ in $\gamma$ satisfies $r(\{e\})=0$.
By Lemmas~\ref{lem:rankssingular} and~\ref{lem:activitiessingular},

\begin{align*}\MoveEqLeft{ \sum_{B\in \mathcal B(Z)} \Bigg(\prod_{\omega \in \inact_{Z,\prec}(B)} \vect x_{B_{\omega}}\Bigg) \Bigg(\prod_{\omega \in \act_{Z,\prec}(B)}
 \Big(\frac {t\vect x_{\underline B_{\omega}}} {|\omega|-1}+\vect x_{B_{\omega}}\Big)\Bigg) }\\
 &= \sum_{B'\in \mathcal B(Z\ba \gamma)} \sum_{u\in \gamma-\{e\}} \Bigg(\prod_{\omega \in \inact_{Z\setminus\gamma,\prec}(B')}  \vect x_{B_{\omega}}\Bigg) \Bigg(\prod_{\omega \in \act_{Z\setminus\gamma,\prec}(B')}
 \Big(\frac {t\vect x_{\underline B'_{\omega}}} {|\omega|-1}+\vect x_{B'_{\omega}}\Big)\Bigg)\\
  &\phantom{=\sum_{B'\in \mathcal B(Z\ba \omega)} \sum_{u\in \gamma-\{e\}}} \cdot
 \Big(\frac {t\vect x_e} {|\gamma|-1}+\vect x_u\Big)\\
 &=\Big(t\vect x_e+\sum_{u\in\gamma-\{e\}} \vect x_u\Big)\\
 &\phantom{=}\cdot \sum_{B'\in \mathcal B(Z\setminus\gamma)}\Bigg(\prod_{\omega \in \inact_{Z\setminus\gamma,\prec}(B')} \vect x_{B_{\omega}}\Bigg) \Bigg(\prod_{\omega \in \act_{Z\setminus\gamma,\prec}(B')}\Big(\frac {t\vect x_{\underline B'_{\omega}}} {|\omega|-1}+\vect x_{B'_{\omega}}\Big)\Bigg).\end{align*}
By using first the induction hypothesis and then Proposition~\ref{prop:delcon}, we see that this is equal to
\[ \Big(t\vect x_e+\sum_{u\in\gamma-\{e\}} \vect x_u\Big) Q(Z\ba \gamma;\vect x',t) = Q(Z;\vect x,t),\]
where $\vect x'$ is the restriction of $\vect x$ to $Z\ba \gamma$.

Now suppose that $\gamma$ is non-singular. Then, by Lemma~\ref{lem:activitiesnonsingular},
\begin{align*} \MoveEqLeft{\sum_{B\in \mathcal B(Z)} \Bigg(\prod_{\omega \in \inact_{Z,\prec}(B)} \vect x_{B_{\omega}}\Bigg) \Bigg(\prod_{\omega \in \act_{Z,\prec}(B)}
 \Big(\frac {t\vect x_{\underline B_{\omega}}} {|\omega|-1}+\vect x_{B_{\omega}}\Big)\Bigg) }\\
 &= \sum_{u\in \gamma} \sum_{B'\in \mathcal B(Z|u)} \vect x_u
\Bigg(\prod_{\omega \in \inact_{Z|u,\prec}(B')}\vect  x_{B'_{\omega}}\Bigg) \Bigg(\prod_{\omega \in \act_{Z|u,\prec}(B')}
 \Big(\frac {t\vect x_{\underline B'_{\omega}}} {|\omega|-1}+\vect x_{B'_{\omega}}\Big)\Bigg).\end{align*}
By using first the induction hypothesis and then Proposition~\ref{prop:delcon}, we see that this is
\[ \sum_{u\in \gamma} \vect x_u Q(Z|u;\vect x',t) = Q(Z;\vect x,t),\] 
where $\vect x'$ is the restriction of $\vect x$ to $Z\ba \gamma$, as required.
\end{proof}

The application of Theorem~\ref{thm:main1} to the unweighted transition polynomial is immediate. 
\begin{corollary}\label{cor:firstimm}
Let $Z$ be a non-degenerate multimatroid and let $\prec$ be a total ordering of its skew classes. Then the unweighted transition polynomial is given by
\[ Q(Z;t)= \sum_{B\in \mathcal B(Z)}  \prod_{\omega \in \act_{\prec}(B)}
 \Big(\frac {t} {|\omega|-1}+1\Big).\]
\end{corollary}

\begin{corollary}\label{cor:secondimm}
Let $Z$ be a $q$-matroid with $q\geq 2$ and let $\prec$ be a total ordering of its skew classes. Then the unweighted transition polynomial is given by
\begin{equation} Q(Z;t)= \sum_{B\in \mathcal B(Z)}   \Big(\frac {t} {q-1}+1\Big)^{|{\act_{\prec}(B)}|}.\label{eq:main1}\end{equation}
\end{corollary}

Substituting $t=q-1$ into~\eqref{eq:main1} and using the definition of $Q(Z;t)$ gives
\begin{equation} \label{eq:cor}
\sum_{T\in \mathcal T(Z)} (q-1)^{n(T)} = \sum_{B\in \mathcal B(Z)} 2^{|{\act_{\prec}(B)}|}.
\end{equation}

\begin{example}[Example~\ref{example} continued]\label{eg:examplectd.1} 
Recall that $Z$, the $3$-matroid described in Example~\ref{example}, has $16$ bases, $10$ transversals with nullity one, and one transversal with nullity two. So its unweighted transition polynomial is $Q(Z;t)=t^2+10t+16t$.

We define a total ordering $\prec$ on the skew classes of $Z$ by requiring that $\omega_a\prec \omega_b\prec \omega_c$.
Then $Z$ has the following four bases having two active skew classes,
with the
active skew classes being $\omega_a$ and $\omega_b$ in each case:
$\{\dott{a}, \dott{b}, \dott{c}\}$,
$\{ \widehat{a}, \dott{b}, \dott{c}\}$,
$\{ \dott{a}, \widehat{b}, \dott{c}\}$ and
$\{ \widehat{a}, \widehat{b}, \dott{c}\}$.
The remaining twelve bases have one active skew class which is $\omega_a$ in each case. 
Thus, by Equation~\eqref{eq:main1}, the unweighted transition polynomial of $Z$ is
\[
  Q(Z;t)= 4\Big(\frac {t} {2}+1\Big)^2+12\Big(\frac {t} {2}+1\Big)=t^2+10t+16.
\]
\expleend\end{example}

\begin{example}[Example~\ref{eg:matroid} continued]\label{eg:matcont}  
In this example we apply Theorem~\ref{thm:main1} to matroids. In particular we show that Theorem~\ref{thm:tutteact} is a special case of Theorem~\ref{thm:main1} and therefore our results generalise the classical matroid results. 
Suppose that $M$ is a matroid with a total ordering $\prec$ of its set $E$ of elements. Then we get a total ordering $\prec$ on the skew classes of $Z(M)$ so that $\omega_e \prec \omega_f$ if and only if $e\prec f$. 

Let $T$ be a transversal of $Z(M)$ and let $A=\phi^{-1}(T)$.
We have
\[ n_{Z(M)}(T) = |T|-r_{Z(M)}(T)
= |A|-r_M(A) + |E-A|-r_{M^*}(E-A).\]
Using the expression for the rank function of a dual matroid,~\cite[Proposition~2.1.9]{MR1207587}, we see that
\[ n_{Z(M)}(T) = |A|-2r_M(A) + r_M(E).\]
Set $x_{\dott e}=u$ and $x_{\overline e} =v$ for every element $e$. Then we have
\begin{equation} \label{eq:mat1} Q(Z(M);\vect x,t) 
= \sum_{T\in \mathcal{T} (\Omega)} t^{n_{Z(M)}(T)} \vect x_T
= \sum_{A\subseteq E} t^{|A|+r_M(E)-2r_M(A)} u^{|A|} v^{|E-A|}.\end{equation}

Now suppose that $T$ is a basis of $Z(M)$ and let $B=\phi^{-1}(T)$.
Recall that the collection of circuits of $Z(M)$ comprises subsets of $\dott E$ corresponding to circuits of $M$  
and subsets of $\overline E$ corresponding to cocircuits of $M$.
Thus a skew class $\omega_e$ 
with $T_{\omega_e}=\dott e$
is active with respect to $T$ if and only if $e$ is internally active with respect to $B$ in $M$, 
and a skew class $\omega_e$ 
with $T_{\omega_e}=\overline e$
is active with respect to $T$ if and only if $e$ is externally active with respect to $B$ in $M$. 
So a skew class $\omega_e$
is active with respect to a basis $T$ of $Z(M)$ if and only if $e$ is internally or
externally active with respect to the corresponding basis $B$ of $M$.
Then Theorem~\ref{thm:main1} applied to $Z(M)$ gives
\begin{align}\label{eq:mat2} 
\begin{split}
\MoveEqLeft Q(Z(M);\vect x,t)\\ &= 
\sum_{B\in \mathcal {B}(M)}
u^{|B|-|{\inta_{\prec}(B)}|}
v^{|E-B|-|{\exta_{\prec}(B)}|}
(tv+u)^{|{\inta_{\prec}(B)}|}
(v+tu)^{|{\exta_{\prec}(B)}|}\\
&=\sum_{B\in \mathcal {B}(M)}
u^{r_M(E)} 
v^{|E|-r_M(E)} 
(tv/u+1)^{|{\inta_{\prec}(B)}|}
(1+tu/v)^{|{\exta_{\prec}(B)}|}.
\end{split}
\end{align}
Setting $u=\sqrt y$, $v=\sqrt x$ and $t=\sqrt{xy}$, and equating the two expressions for $Q(Z(M))$ given in Equations~\eqref{eq:mat1} and~\eqref{eq:mat2}, gives for $x\geq 0$ and $y\geq 0$,
\begin{align*} 
&\sum_{B\in \mathcal{B}(M)} {\sqrt y}^{\,r_M(E)} {\sqrt x}^{\,|E|-r_M(E)} (x+1)^{|{\inta_{\prec} (B)}|} (y+1)^{|{\exta_{\prec} (B)}|} \\
 &= \sum_{A\subseteq E} {\sqrt y}^{\,r_M(E)} {\sqrt x}^{\,|E|-r_M(E)} x^{r_M(E)-r_M(A)} y^{|A|-r_M(A)}.
\end{align*}

But this is a polynomial identity, so it holds for all $x$ and $y$. 
 From this we deduce that 
 \[ T(M;x+1,y+1) = \sum_{B\in \mathcal{B}(M)}  (x+1)^{|{\inta_{\prec} (B)}|} (y+1)^{|{\exta_{\prec} (B)}|}.\]
So we see that Theorem~\ref{thm:tutteact} is a special case of Theorem~\ref{thm:main1}. 
The argument above also shows that 
\begin{equation}\label{eq:tuttetrans} T(M;x+1,y+1) =  \sqrt{x}^{\,r_M(E)-|E|} \sqrt{y}^{\,-r_M(E)} Q(Z(M);\vect x, t), \end{equation}
where $t=\sqrt{xy}$, $x_{\dott e}=\sqrt{y}$ and $x_{\overline e}=\sqrt{x}$. 
When $Q(Z(M);\vect x,t)$ is expanded in powers of $x_{\dott e}$, $x_{\overline e}$, and $t$, the sum of the degrees of $x_{\dott e}$ and $x_{\overline e}$ in each monomial is equal to $|E|$ and the degree of $x_{\dott e}$ in each monomial containing $t^0$ is equal to $r_M(E)$.
Thus the weighted transition polynomial generalizes the Tutte polynomial, in the sense that the Tutte polynomial may be obtained from $Q(Z(M))$ by an appropriate symbolic substitution.
\expleend\end{example}

\subsection{A combinatorial interpretation}\label{ss:results2}
Our second main result gives a combinatorial interpretation of~\eqref{eq:cor}. When applied to $Z(M)$ for some matroid $M$, it essentially specializes to 
Theorem~\ref{thm:partition}, in a sense that we make precise in Example~\ref{eg:matcont1.5}.

Let $Z$ be a non-degenerate multimatroid having carrier $(U,\Omega)$ and let $\prec$ be a total ordering of its skew classes. Given a basis $B$ of $Z$, let
\begin{align*}
B_{\insa} := \{B_{\omega} : \omega \in \act_{\prec}(B)\},\\
\shortintertext{and}
B_{\outa} := \{\underline B_{\omega} : \omega \in \act_{\prec}(B)\}.
\end{align*}
Again, we require non-degeneracy here to guarantee that $B$ is a transversal so that $B_{\omega}$ is defined for every skew class $\omega$.
An element of $B_{\insa}$, respectively $B_{\outa}$, is said to be \emph{inside active}, respectively \emph{outside active}, with respect to $B$.\footnote{We have deliberately avoided using the terms internally and externally active, because the concepts we are defining here are not analogous to the usual notions of internally and externally active in matroids.}

Now let \[ H_{Z,\prec}(B) := \{ T\in \mathcal T(\Omega): B-B_{\insa} \subseteq T \subseteq B \cup B_{\outa}\}.\]
Thus, each element of $H_{Z,\prec}(B)$ is obtained from $B$ by removing some (or possibly none) of its inside active elements and replacing them with the
corresponding outside active elements from the same skew classes. We let $\mathcal H_{Z,\prec}:= \{H_{Z,\prec}(B) : B\in \mathcal B(Z)\}$. Note that $|H_{Z,\prec}(B)|=2^{|{\act_{\prec}(B)}|}$. 
So, from Equation~\eqref{eq:cor} we get 
\[\sum_{T\in \mathcal {T}(Z)} (q-1)^{n(T)}=\sum_{B\in \mathcal B(Z)}|H_{Z,\prec}(B)|=\sum_{H\in \mathcal H_{Z,\prec}}|H|.\]

To state the next results we need the following definition. Given a subtransversal $S$ of a multimatroid $Z$, define
\[ \mce_{Z}(S,\prec) = \{\min(C): C\subseteq S, \text{ and } C\in \mathcal C(Z)\}.\]
Then we define $\mcs_{Z}(S,\prec)$ to be the collection of skew classes meeting $\mce_Z(S,\prec)$. In both cases, we omit the dependence on $Z$ when this is obvious. Thus $\mce_{Z}(S,\prec)$ is the collection of elements of $S$ which are the least element of a circuit $C\subseteq S$ and $\mcs_{Z}(S,\prec)$ is the corresponding collection of skew classes. Obviously, 
$|{\mce_{Z}(S,\prec)}|$ $=$ $|{\mcs_{Z}(S,\prec)}|$.

The following lemma is surely known, but we cannot find an explicit reference. We use $\cl$ to denote the closure operator of a matroid~\cite{MR1207587}.

\begin{lemma}\label{lem:matroidmincircuit}
Let $M$ be a matroid with a total ordering $\prec$ of its elements and let $S$ be a subset of its elements. Then $n(S)=|{\mce_{Z}(S,\prec)}|$.
\end{lemma}

\begin{proof}
$S-\mce(S,\prec)$ does not contain any circuits, so it is independent. Hence $n(S)\leq |{\mce(S,\prec)}|$. On the other hand, $S\subseteq \cl(S-\mce(S,\prec))$ so $r(S)\leq r(S-\mce(S,\prec))$. Hence $n(S)\geq |{\mce(S,\prec)}|$. So the result follows.
\end{proof}

Therefore if $Z$ is a multimatroid and $\prec$ is a total ordering of its skew classes, then for every subtransversal $S$ of $Z$, we have $n(S)=|{\mce(S,\prec)}|$.

We now state and prove our second main result.
\begin{theorem}\label{thm:maingen2}
Let $Z=(U,\Omega,r)$ be a non-degenerate multimatroid and let $\prec$ be a total ordering of its skew classes. Then each transversal $T$ of $\Omega$ appears in exactly 
$\prod_{\omega \in \mcs(T,\prec)} (|\omega|-1)$ members of $\mathcal H_{Z,\prec}$.
\end{theorem}

\begin{proof}
We proceed by induction on $|\Omega|$. If $|\Omega|=0$, then the result is immediate. So we may suppose that $|\Omega|>0$. Let $\gamma$ be the greatest skew class according to $\prec$.

Suppose first that $\gamma$ is singular and the element $e$ in $\gamma$ satisfies $r(e)=0$. Let $B$ be a basis of $Z$. By Lemma~\ref{lem:activitiessingular}, $\act_{Z,\prec}(B)=\act_{Z\setminus\gamma,\prec}(B-\gamma)\cup \{\gamma\}$. Therefore
\begin{equation} \label{eq:HZ} H_{Z,\prec}(B) = \{ S\cup \{B_{\gamma}\}, S\cup \{e\}: S\in H_{Z\ba \gamma}(B-\gamma)\}.\end{equation}
Now let $T$ be a transversal of $\Omega$. If $T_\gamma\ne e$, then by Lemma~\ref{lem:circuits},
$\mcs_Z(T,\prec)=\mcs_{Z\setminus \gamma}(T-\gamma,\prec)$.
Using the inductive hypothesis, $T-\gamma$ appears in exactly 
\[\prod_{\omega \in \mcs_{Z\setminus \gamma}(T-\gamma,\prec)} (|\omega|-1)\] members of
$\mathcal H_{Z|T_{\gamma},\prec}$. So, by Equation~\eqref{eq:HZ}, $T$ appears in exactly  
\[\prod_{\omega \in \mcs_{Z\setminus \gamma}(T-\gamma,\prec)} (|\omega|-1)
=\prod_{\omega \in \mcs_{Z}(T,\prec)} (|\omega|-1)\]
members of $\mathcal H_{Z,\prec}$. Now suppose that $T_{\gamma}=e$. By Lemma~\ref{lem:circuits}, 
$\mcs_Z(T,\prec)=\mcs_{Z\setminus \gamma}(T-\gamma,\prec)\cup \{\gamma\}$.
Using the inductive hypothesis, $T-\{e\}$ appears in exactly 
\[\prod_{\omega \in \mcs_{Z\setminus \gamma}(T-\gamma,\prec)} (|\omega|-1)\]
members of
$\mathcal H_{Z\ba \gamma,\prec}$. For every $y$ in $\gamma -\{e\}$, and every basis $B$ of $Z\ba \gamma$, it follows from Equation~\eqref{eq:HZ} that for a transversal $S$ of $Z\ba \gamma$,
$S \in H_{Z\setminus\gamma}(B)$ if and only if $S\cup \{e\} \in H_{Z}(B\cup \{y\})$. Hence $T$ appears in exactly 
\[(|\gamma|-1) \prod_{\omega \in \mcs_{Z\setminus \gamma}(T-\gamma,\prec)} (|\omega|-1)
=\prod_{\omega \in \mcs_{Z}(T,\prec)} (|\omega|-1)\]
members of $\mathcal H_{Z,\prec}$, as required.

Now suppose that $\gamma$ is non-singular. Let $T$ be a transversal of $\Omega$ and let $x:=T_{\gamma}$. We claim that $\mcs_{Z|x}(T-\{x\},\prec)=\mcs_Z(T,\prec)$. To see this, suppose that $\omega \in \mcs_{Z|x}(T-\{x\},\prec)$. Then $Z|x$ has a circuit $C$ with $C\subseteq T-\{x\}$ and $\min (C)\in \omega$. 
Using Lemma~\ref{lem:usefulnull}, we deduce that either $C$ or $C\cup\{x\}$ is a circuit of $Z$. As $x\in \gamma$ and $\gamma$ is the greatest skew class according to $\prec$, we deduce that \mbox{$\omega \in \mcs_Z(T,\prec)$}. 
So $\mcs_{Z|x}(T-\{x\},\prec)\subseteq \mcs_Z(T,\prec)$.
Using Lemma~\ref{lem:usefulnull} once more, we see that $n_{Z|x}(T-\{x\})=n_Z(T)$. Hence, by Lemma~\ref{lem:matroidmincircuit}, $|{\mcs_{Z|x}(T-\{x\},\prec)}|=|{\mcs_Z(T,\prec)}|$, and the claim follows.

Using the inductive hypothesis, $T-\{x\}$ is contained in exactly 
\[ \prod_{\omega \in \mcs_{Z|x}(T-\{x\},\prec)} (|\omega|-1)
=\prod_{\omega \in \mcs_{Z}(T,\prec)} (|\omega|-1)\] members of the collection
$\mathcal H_{Z|x},\prec$. As $\gamma$ is non-singular and $\gamma$ is the greatest skew class according to $\prec$, it is not active with respect to $B$. So, by Lemma~\ref{lem:activitiesnonsingular},
 \[ H_{Z,\prec}(B) = \{S\cup\{x\} : S\in H_{Z|x}(B-\{x\})\}.\]
 Thus $T$ is contained in exactly 
$\prod_{\omega \in \mcs_{Z}(T,\prec)} (|\omega|-1)$ members of $\mathcal H_{Z,\prec}$, as required.
\end{proof}

The next corollary follows immediately from Lemma~\ref{lem:matroidmincircuit} and Theorem~\ref{thm:maingen2}.
\begin{corollary}
\label{cor:main2}
Let $Z=(U,\Omega,r)$ be a $q$-matroid with $q\geq 2$ and let $\prec$ be a total ordering of its skew classes. Then each transversal $T$ of $\Omega$ appears in exactly $(q-1)^{n(T)}$ members of $\mathcal H_{Z,\prec}$.
\end{corollary}

\begin{example}[Examples~\ref{example} and~\ref{eg:examplectd.1} continued]\label{eg:examplectd.2}  We again define a total ordering $\prec$ on the skew classes of $Z$, the $3$-matroid described in Example~\ref{example}, by requiring that $\omega_a\prec \omega_b\prec \omega_c$.
Recall that $Z$ has the following four bases having two active skew classes,
with the
active skew classes being $\omega_a$ and $\omega_b$ in each case:
$\{\dott{a}, \dott{b}, \dott{c}\}$,
$\{ \widehat{a}, \dott{b}, \dott{c}\}$,
$\{ \dott{a}, \widehat{b}, \dott{c}\}$ and
$\{ \widehat{a}, \widehat{b}, \dott{c}\}$. The  one dependent transversal
$\{ \overline{a}, \overline{b}, \dott{c}\}$ with nullity two belongs to the set $H_{Z,\prec}(B)$ for each of these four bases $B$.
The remaining twelve bases have one active skew class which is $\omega_a$ in each case. The ten dependent transversals of nullity one belong to exactly two of the sets in $\mathcal H_{Z,\prec}$. For example, the dependent transversal $\{ \overline{a}, \overline{b}, \overline{c}\}$ belongs to
$H_{Z,\prec}(\{ \dott{a}, \overline{b}, \overline{c}\})$ and $H_{Z,\prec}(\{ \widehat{a}, \overline{b}, \overline{c}\})$.
\expleend\end{example}

The special case of Corollary~\ref{cor:main2} when $q=2$ shows that the members of $\mathcal H_{Z,\prec}$ partition $\mathcal T(\Omega)$. 
This extends Theorem~\ref{thm:partition} as the next example illustrates.

\begin{example}[Examples~\ref{eg:matroid} and~\ref{eg:matcont} continued]\label{eg:matcont1.5}
In this example we apply Theorem~\ref{thm:maingen2} to matroids to show that Theorem~\ref{thm:partition} is a special case of Theorem~\ref{thm:maingen2} and therefore that this result also generalizes classical matroid results. 
As in Example~\ref{eg:matcont} we suppose that $M$ is a matroid with a total ordering $\prec$ of its set $E$ of elements, and also use $\prec$ to denote the corresponding total ordering of the skew classes of $Z(M)$.

Let $B$ be a basis of $M$. Then
\[ H_{Z(M),\prec}(\phi(B)) := \{ T\in \mathcal T(\Omega): \phi(B)-\phi(B)_{\insa} \subseteq T \subseteq \phi(B) \cup \phi(B)_{\outa}\}.\]
As we observed in Example~\ref{eg:matcont}, a skew class $\omega_e$
is active with respect to $\phi(B)$ in $Z(M)$ if and only if $e$ is internally or
externally active with respect to $B$ in $M$.
Hence
\begin{align*} \phi(B)_{\insa} &= \{\dott e: e\in \inta_{M,\prec}(B)\} 
\cup \{\overline e: e\in \exta_{M,\prec} (B)\}\\
\shortintertext{and}
\phi(B)_{\outa} &= \{\overline e: e\in \inta_{M,\prec} (B)\} 
\cup \{\dott e: e\in \exta_{M,\prec} (B)\}.\end{align*}
Thus $\phi^{-1}(H_{Z(M),\prec}(\phi(B))) = [B-\inta_{M,\prec}(B),B\cup \exta_{M,\prec}(B)]$.
Because $Z(M)$ is a $2$-matroid, Corollary~\ref{cor:main2} implies that the collection $\mathcal H_{Z(M),\prec}$ partitions $\mathcal T(Z(M))$. As $\phi$ is a bijection (both considered as a map from $2^E$ to $\mathcal T(Z(M))$ and as a map from $\mathcal B(M)$ to $\mathcal B(Z(M))$), this implies that the collection of Boolean intervals \[\{[B-\inta_{M,\prec}(B), B\cup \exta_{M,\prec}(B)]: B\in \mathcal B(M)\}\] partitions $2^E$, which is the conclusion of Theorem~\ref{thm:partition}.
\expleend\end{example}

\section{Compatible transversals}\label{s:kochol}

In this section we extend from matroids to multimatroids a more recent expression for the Tutte polynomial introduced by Kochol~\cite{zbMATH07347313}. We begin by extending Kochol's ideas to define a notion of compatibility for the transversals of a multimatroid. This leads to our third main result, Theorem~\ref{thm:cocomp} which gives an expression for the transition polynomial as a sum over compatible transversals.
In the second half of the section we explore the links between compatibility and the collection $\mathcal H_{Z,\prec}$.
This suggests an equivalence relation on the bases of a non-degenerate multimatroid leading to our fourth main result, Theorem~\ref{thm:equiv}, which in turn leads to another expression for the transition polynomial.

\subsection{A compatible transversals expansion of the transition polynomial}

Given a matroid $M$ with a total ordering $\prec$ of its elements, Kochol~\cite{zbMATH07347313} defines a subset $A$ of the elements of $M$ to be $(M,\prec)$-\emph{compatible} if and only if $M$ has no circuit $C$ such that $A \cap C= \{\min (C)\}$. The set $\mathcal D(M,\prec)$ is then defined as follows.
\begin{multline*} \mathcal D(M,\prec) \\:= \{A\subseteq  E(M): \text{$A$ is $(M^\ast,\prec)$-compatible and $E(M)-A$ is $(M,\prec)$-compatible}\}.\end{multline*}
This leads to the following expression for the Tutte polynomial of $M$.
\begin{theorem}[Kochol~\cite{zbMATH07347313}]\label{thm:Tuttecompat}
Let $M=(E,r)$ be a matroid and let $\prec$ be a total ordering on $E$. Then
    \[T(M;x,y) = \sum_{A \in \mathcal D(M,\prec) } x^{r(E)-r(A)}y^{|A|-r(A)}.\]
\end{theorem}
Substituting $x=y=1$ shows that
$|\mathcal D(M,\prec)|$ is equal to the number of bases of $M$. 
Kochol's original proof of Theorem~\ref{thm:Tuttecompat} used the deletion-contraction formula for the Tutte polynomial, but later proofs by Kochol~\cite{zbMATH07725327} and Pierson~\cite{zbMATH07814958} have used the activities expansion from Theorem~\ref{thm:tutteact} and established a bijection between the collection of bases and $\mathcal D(M,\prec)$. We shall introduce the notion of a compatible transversal of a multimatroid $Z$ and a mapping from $\mathcal T(Z)$ to the set of compatible transversals of $Z$. For a matroid $M$ we shall see in Example~\ref{eg:matcont3} that this mapping restricts to a bijection from $\mathcal B(Z(M))$ to the set of compatible transversals equivalent to that of Kochol~\cite{zbMATH07725327} and Pierson~\cite{zbMATH07814958} via the 
bijection $\phi$ described in Example~\ref{eg:matroid}.
The form of $\mathcal D(M,\prec)$ suggests that it is particularly amenable to an extension to multimatroids, as it mentions both the circuits and cocircuits of $M$, that is, the entire collection of circuits of the $2$-matroid $Z(M)$. 

\begin{definition}Given a multimatroid $Z$ and a total ordering $\prec$ of its skew classes, we say that a transversal $T$ of $Z$ is $(Z,\prec)$-\emph{compatible} if there is no circuit $C$ of $Z$ such that $C-T=\{\min(C)\}$. 
\end{definition}
We warn the reader that our concept of compatibility does not directly generalize that introduced by Kochol, in the sense that the compatible sets of a matroid $M$ do not coincide with those of $M$ considered as a $1$-matroid, nor with those of the $2$-matroid $Z(M)$. Instead, there is a close relationship between $\mathcal D(M,\prec)$ and compatible sets of $Z(M)$, as illustrated in the following example.

\begin{example}[Examples~\ref{eg:matroid},~\ref{eg:matcont}, and~\ref{eg:matcont1.5} continued]\label{eg:matcont2}     Let $M$ be a matroid with a total ordering $\prec$ of its set $E$ of elements. Let $\prec$ also denote the total ordering induced on the skew classes of $Z(M)$.
    Then a subset $A$ of $E$ is in $\mathcal D(M,\prec)$ if and only if $E-A$ is $(M,\prec)$-compatible and $A$ is $(M^*,\prec)$-compatible. The former occurs if and only if there is no circuit $C$ of $M$ with $(E-A)\cap C =\{\min(C)\}$ or equivalently $C-A=\{\min (C)\}$. This happens if and only if $Z(M)$ has no circuit $C$ with $C\subseteq \dott E$ such that 
$C-\dott A=\{\min (C)\}$. The latter occurs if and only if there is no cocircuit $C^*$ of $M$ with $A\cap C^*= \{\min(C^*)\}$. This happens if and only if $Z(M)$ has no circuit $C$ with $C\subseteq \overline E$ such that $\overline A\cap C= \{\min(C)\}$ or equivalently $C - \overline {(E-A)} = \{\min(C)\}$. Combining these conditions we see that a subset $A$ of $E(M)$ is in $\mathcal D(M,\prec)$ if and only if 
$\phi(A)$ is $(Z(M),\prec)$-compatible.
\expleend\end{example}

Consider the following process to construct a transversal $T$ of $Z$, element by element. First, start with $T=\emptyset$. Then at each stage take the greatest skew class $\omega$ according to $\prec$ with $T\cap \omega = \emptyset$. If there is an element $e$ of $\omega$ such that $r(T\cup\{e\})=r(T)$ then add $e$ to $T$. Note that, by Lemma~\ref{lem:rankssingular}, if such an element $e$ exists then it must be unique. Otherwise, add any element of $\omega$ to $T$ and repeat until $T$ is a transversal. Then this process always leads to a $\prec$-compatible transversal, and by making appropriate choices may yield any $\prec$-compatible transversal.

\begin{example}[Examples~\ref{example},~\ref{eg:examplectd.1}, and~\ref{eg:examplectd.2} continued]\label{eg:examplectd.3}
Recall that $Z$, the $3$-matroid described in Example~\ref{example}, has the following seven circuits: 
$\{ \overline{a}, \overline{b}\}$,
$\{ \overline{a}, \dott {c}\}$,
$\{ \overline{b}, \dott {c}\}$,
$\{ \dott {a}, \dott {b}, \overline{c}\}$,
$\{ \widehat{a}, \widehat{b}, \overline{c}\}$,
$\{ \widehat{a}, \dott {b}, \widehat{c}\}$ and
$\{ \dott {a}, \widehat{b}, \widehat{c}\}$.
We again define a total ordering $\prec$ on the skew classes of $Z$, by requiring that $\omega_a\prec \omega_b\prec \omega_c$.
Then the $\prec$-compatible transversals of $Z$ are
$\{ \overline{a}, \overline{b}, \dott{c}\}$,
$\{ \overline{a}, \overline{b}, \overline{c}\}$,
$\{ \dott{a}, \dott{b}, \overline{c}\}$,
$\{ \widehat{a}, \widehat{b}, \overline{c}\}$,
$\{ \overline{a}, \overline{b}, \widehat{c}\}$,
$\{ \widehat{a}, \dott{b}, \widehat{c}\}$,
$\{ \dott{a}, \widehat{b}, \widehat{c}\}$.
In $\{ \overline{a}, \overline{b}, \dott{c}\}$ the choices of both $\overline{a}$ and $\overline{b}$ were forced by the choice of $\dott{c}$, whereas in the other six $\prec$-compatible transversals only the choice of element from $\omega_a$ was forced by the choices of elements from $\omega_b$ and $\omega_c$. Observe that changing every such forced element to any one of the other elements from its skew class gives a basis and every basis can be constructed in this way. The fact that this phenomenon always works leads us to Theorem~\ref{thm:equiv} near the end of this section. 
\expleend\end{example}

Let $Z$ be a multimatroid with a total ordering $\prec$ of its skew classes, and let $T$ be a transversal of $Z$. Recall that a skew class $\omega$ is active with respect to $T$ if there is a circuit $C$ of $Z$ with $\min (C)\in \omega$ and $C-\omega \subseteq T$.

The following lemmas will be crucial. The first is a straightforward consequence of Proposition~\ref{prop:uniqueskewpair}.

\begin{lemma}\label{lem:compatunique}
    Let $Z$ be a multimatroid with a total ordering $\prec$ of its skew classes, and let $T$ be a transversal of $Z$. Let $\omega$ be a skew class of $Z$, and $C_1$ and $C_2$ be circuits of $Z$ such that for $i\in\{1,2\}$, $\min (C_i)  \in\omega$ and $C_i-\omega\subseteq T$. Then $\min(C_1)=\min(C_2)$. 
\end{lemma}

\begin{lemma}\label{lem:compatcruc}
    Let $Z$ be a multimatroid with a total ordering $\prec$ of its skew classes, and let $T$ be a transversal of $Z$. If $\omega$ is active with respect to $T$, then $Z$ has a circuit $C$ such that $\min (C) \in \omega$, $C-\omega \subseteq T$ and $C$ meets no skew class other than $\omega$ that is active with respect to $T$.
\end{lemma}

\begin{proof}
    Suppose for contradiction that there is a skew class $\omega$ which is active with respect to $T$, and every circuit $C$ with $\min (C)\in\omega$ and $C-\omega \subseteq T$ meets some active skew class other than $\omega$. Choose $\omega$ to be the greatest skew class for which this is true and let $C'$ be a circuit with $\min(C')\in \omega$, $C'-\omega\subseteq T$ meeting as few active skew classes with respect to $T$ as possible. (Note that $C'$ 
exists because $\omega$ is active).
    Then $C'$ must meet at least one skew class $\omega'$ which is active with respect to $T$ and satisfies $\omega\prec \omega'$. By the maximality of $\omega$, there is a circuit $C''$ with $\min (C'')\in \omega'$, $C''-\omega'\subseteq T$ such that $C''$ meets no active skew class other than $\omega'$. By Proposition~\ref{prop:uniqueskewpair}, $C'_{\omega'} = C''_{\omega'}$, so $C'\cup C''$ is a subtransversal of $Z$.
    Hence the strong circuit elimination axiom of matroids~\cite[Proposition~1.4.12]{MR1207587}, implies that there is a circuit $C'''$ of $Z$ with $C'_{\omega} \in C''' \subseteq (C'\cup C'')-\{C'_{\omega'}\}$. Then $\min(C''')=\min (C')\in \omega$, $C'''-\omega \subseteq T$ and $C'''$ meets fewer active skew classes than $C'$ contradicting the definition of $C'$. So the result follows.  
\end{proof}

As a consequence of Lemma~\ref{lem:compatunique}, if $\omega$ is active with respect to the transversal $T$, then we may define 
 $\underline T_{\omega}$ to denote the common element $\min(C)$ of every circuit $C$ with $\min(C)\in \omega$ and $C-\omega \subseteq T$. 
 (This extends our earlier definition of $\underline B_{\omega}$ when $B$ is a basis.) When $T$ is not a basis it is possible that  
$\underline T_{\omega} \in T$.

The next lemma shows that if we change a transversal by replacing one of its elements from an active skew class by another from that skew class, then we do not change the collection of active skew classes. We recall that $A\bigtriangleup B$ denotes the \emph{symmetric difference}  $(A\cup B) - (A \cap B)$ of the sets $A$ and $B$.
\begin{lemma}\label{lem:switchskew}
    Let $Z$ be a multimatroid with a total ordering $\prec$ of its skew classes, and let $T$ be a transversal of $Z$.
    Suppose that $\omega$ is active with respect to $T$ and that $\{T_\omega,y\}$ is a skew pair. Let $T':=T\bigtriangleup \{T_\omega,y\}$. Then $\act_{Z,\prec}(T)=\act_{Z,\prec}(T')$. Moreover, if $\omega' \in \act_{Z,\prec}(T)$, then 
    $\underline T_{\omega'}=\underline T'_{\omega'}$.
\end{lemma}

\begin{proof}
    Suppose that $\omega'$ is active with respect to $T$. Then by Lemma~\ref{lem:compatcruc}, there is a circuit $C$ such that $\min(C) \in \omega'$ and $C-\omega'\subseteq T$, meeting no skew class other than $\omega'$ that is active with respect to $T$. Then $C-\omega'\subseteq T'$, so $\omega'\in \act_{Z,\prec}(T')$ and $\underline T_{\omega'}=\underline T'_{\omega'}$. 
    
    Now suppose that $\omega'$ is active with respect to $T'$. We may assume that $\omega'\ne \omega$. 
Then by Lemma~\ref{lem:compatcruc}, there is a circuit $C$ such that $\min(C) \in \omega'$ and $C-\omega'\subseteq T'$, meeting no skew class other than $\omega'$ that is active with respect to $T'$. So, from the first part, $C$ meets no skew class other than $\omega'$ that is active with respect to $T$.
Thus, in particular, $C\cap \omega=\emptyset$ and $C-\omega'\subseteq T$. Hence $\omega'\in \act_{Z,\prec}(T)$.
\end{proof}
    
We now see how to associate any transversal of a multimatroid $Z$ with 
a $(Z,\prec)$-compatible transversal. Given a transversal $T$ of $Z$ we define its $\prec$-\emph{canonical compatible transversal}, denoted by $\ccl_{\prec}(T)$ to be the transversal determined by setting $(\ccl_{\prec}(T))_{\omega}$ as follows.
\[ (\ccl_{\prec}(T))_{\omega}  = 
\begin{cases}
\underline T_{\omega}
& \text{if $\omega$ is active,}\\
T_{\omega} & \text{if $\omega$ is inactive.}
\end{cases}\]
We will prove that $\ccl_{\prec}(T)$ is $(Z,\prec)$-compatible in Proposition~\ref{prop:propsofcompat}.

The next result is a straightforward consequence of Lemma~\ref{lem:switchskew}.
\begin{lemma}\label{lem:cclswitchskew}
    Let $Z$ be a multimatroid with a total ordering $\prec$ of its skew classes, and let $T$ be a transversal of $Z$.
    Suppose that $\omega$ is active with respect to $T$ and that $\{T_\omega,y\}$ is a skew pair. Let $T':=T\bigtriangleup \{T_\omega,y\}$. Then $\ccl_{\prec}(T)=\ccl_{\prec}(T')$. 
\end{lemma}

We now give the key properties of $(Z,\prec)$-canonical compatible transversals.

\begin{proposition}\label{prop:propsofcompat}
    Let $Z$ be a multimatroid with a total ordering $\prec$ of its skew classes and let $T$ be a transversal of $Z$. Then each of the following holds.
\begin{enumerate}
\item\label{prop:propsofcompat.1}  $\ccl_{\prec}(T)$ and $T$ have the same active skew classes.
\item \label{prop:propsofcompat.2}
$\ccl_{\prec}(\ccl_{\prec}(T))=\ccl_{\prec}(T)$.
\item\label{prop:propsofcompat.3} $\ccl_{\prec}(T)$ is $(Z,\prec)$-compatible.
\end{enumerate}
\end{proposition}

\begin{proof}
The first and second parts follow using induction on the number of active skew classes with respect to $T$ and Lemmas~\ref{lem:switchskew} and~\ref{lem:cclswitchskew}, respectively. 

To prove the third part, suppose that $C$ is a circuit of $Z$ with $C-\{\min(C)\} \subseteq \ccl_{\prec}(T)$. Let $\omega$ be the skew class containing $\min (C)$. Then $\omega$ is active with respect to $\ccl_{\prec}(T)$. By the second part, $\min(C) = 
(\ccl_{\prec}(T))_{\omega}$. Hence $C\subseteq \ccl_{\prec}(T)$, as required.
\end{proof}

The next result is an easy consequence of the previous one.
\begin{proposition}\label{prop:compatclos} Let $Z$ be a multimatroid with a total ordering $\prec$ of its skew classes and let $T$ be a transversal of $Z$.
Then $T$ is $(Z,\prec)$-compatible if and only if $\ccl_{\prec}(T)=T$.
\end{proposition}

\begin{proof}
    Suppose that $T$ is $(Z,\prec)$-compatible. Then there is no circuit $C$ such that $C-T=\{\min(C)\}$. So for every skew class $\omega$ which is active with respect to $T$, we have $T_{\omega}=\underline{T}_{\omega}$. Hence $\ccl_{\prec}(T)=T$. 

    Conversely, suppose that $T=\ccl_{\prec}(T)$. By Proposition~\ref{prop:propsofcompat} Part~\ref{prop:propsofcompat.3}, $\ccl_{\prec}(T)$ is $(Z,\prec)$-compatible, so $T$ is too.
\end{proof}

\begin{proposition}\label{prop:compatcriterion}
Let $Z$ be a multimatroid with a total ordering $\prec$ of its skew classes 
and let $T$ and $T'$ be transversals of $Z$. Then both the following hold.
\begin{enumerate}
    \item \label{prop:compatcriterion.1} $\ccl_{\prec}(T')=\ccl_{\prec}(T)$ if and only if $T'_{\omega} = T_{\omega}$ for every skew class $\omega$ that is inactive with respect to $T$.
    \item \label{prop:compatcriterion.2} If $T'_{\omega} = T_{\omega}$ for every skew class $\omega$ that is inactive with respect to $T$, then $\act_{\prec}(T') =\act_{\prec}(T)$.
\end{enumerate}
 Moreover when $T$ is $(Z,\prec)$-compatible we have $\ccl_{\prec}(T')=T$ if and only if 
    $T'_{\omega} = T_{\omega}$ for every skew class $\omega$ that is inactive with respect to $T$.
\end{proposition}
\begin{proof}
Suppose first that $T'_{\omega} = T_{\omega}$ for every skew class $\omega$ that is inactive with respect to $T$. We prove by induction on $|T' \bigtriangleup T|$ that $\ccl_{\prec}(T)=\ccl_{\prec}(T')$ 
and $\act_{\prec}(T')=\act_{\prec}(T)$.
If $|T' \bigtriangleup T|=0$, then $T'=T$ and there is nothing to prove. Suppose then that $|T' \bigtriangleup T|>0$. Let $\omega$ be a skew class such that $T'_{\omega}\ne T_{\omega}$. Then $\omega$ is active with respect to $T$. Let $T'':=T\bigtriangleup \{T_{\omega}, T'_{\omega}\}$. By Lemma~\ref{lem:switchskew} $\act_{\prec}(T'')=\act_{\prec}{T}$ and by Lemma~\ref{lem:cclswitchskew} $\ccl_{\prec}(T'')=\ccl_{\prec}(T)$. Thus we may apply the inductive hypothesis to $T''$ and $T'$ to deduce that $\act_{\prec}(T)=\act_{\prec}(T'')=\act_{\prec}(T')$ and $\ccl_{\prec}(T)=\ccl_{\prec}(T'')=\ccl_{\prec}(T')$, establishing one direction of the first part and the second part.

For the converse of the first part, suppose that $\ccl_{\prec}(T)=\ccl_{\prec}(T')$. Let $T'':=\ccl_{\prec}(T)$.
Then by applying Proposition~\ref{prop:propsofcompat} Part~\ref{prop:propsofcompat.1} twice, $\act_{Z,\prec}(T)=\act_{Z,\prec}(T'')=\act_{Z,\prec}(T')$, so it follows from the definition of the canonical compatible closure that $T_{\omega}=T'_{\omega}$ for every skew class $\omega$ that is inactive with respect to $T$.

The last part follows immediately from the first part and Proposition~\ref{prop:compatclos}.
\end{proof}

We next establish a simple expression for the nullity of a transversal $T$ depending on $\ccl_{\prec}(T)$.

\begin{proposition}\label{prop:nulltranscompat}
Let $Z$ be a multimatroid, let $\prec$ be a total ordering of its skew classes and let $T$ be a transversal of $Z$. Then
\[ n(T) = |{\act_{\prec}(\ccl_{\prec}(T))}| - |T\bigtriangleup \ccl_{\prec}(T)|.\]
\end{proposition}

\begin{proof}
Note that $0\leq |{T\bigtriangleup \ccl_{\prec}(T)}| \leq |{\act_{\prec}(\ccl_{\prec}(T))}|$.
Suppose that $T=\ccl_{\prec}(T)$. Then by Proposition~\ref{prop:compatclos}, $T$ is $(Z,\prec)$-compatible. 
We claim that 
\[ \act_{\prec}(T)=\mcs(T,\prec).\] To show this, suppose first that $\omega \in \act_{\prec}(T)$. Then there is a circuit $C$ of $Z$ with $\min(C)\in \omega$ and $C-\omega\subseteq T$. But as $T$ is $\prec$-compatible, $C\subseteq T$ and $\omega\in \mcs(T,\prec)$. Conversely, if $\omega \in\mcs(T,\prec)$, then there is a circuit $C$ with $C\subseteq T$ and $\min (C)\in \omega$, so $\omega \in \act_{\prec}(T)$. 
Thus, by Lemma~\ref{lem:matroidmincircuit}, 
\[ n(T)=|{\act_{\prec}(T)}| = |{\act_{\prec}(\ccl_{\prec}(T))}| - |{T\bigtriangleup \ccl_{\prec}(T)}|.\]

Next suppose that $|{T\bigtriangleup \ccl_{\prec}(T)}|=|{\act_{\prec}(\ccl_{\prec}(T))}|$. Then for every circuit $C$ with $C-\min(C)\subseteq T$, we have $\min(C) \notin T$. Thus $T$ is a basis of $Z$. Hence 
\[ n(T)=0 = |{\act_{\prec}(\ccl_{\prec}(T))}| - |{T\bigtriangleup \ccl_{\prec}(T)}|.\]
So the expression for $n(T)$ holds in the cases when $T=\ccl_{\prec}(T)$ or when 
$|{T\bigtriangleup \ccl_{\prec}(T)}|$ $=$ $|{\act_{\prec}(\ccl_{\prec}(T))}|$.

Axiom~\ref{axiom1} implies that if $T_1$ and $T_2$ are transversals 
differing on at most one skew class, then $|n(T_1)-n(T_2)|\leq 1$. The result now follows.
\end{proof}

It is now easy to obtain an expression for the transition polynomial as a sum over $\prec$-compatible transversals. Given a multimatroid $Z$ with a total ordering $\prec$ of its skew classes, we let $\Comp(Z,\prec)$ denote the collection of $\prec$-compatible transversals of $Z$. 

\begin{theorem}\label{thm:cocomp}
Let $Z$ be a multimatroid and let $\prec$ be a total ordering of its skew classes. Then
\[Q(Z;\vect x,t)= \sum_{T\in \Comp(Z,\prec)}
\Bigg(\prod_{\omega \in \inact_{\prec}(T)} \vect x_{T_{\omega}}\Bigg) 
\Bigg(\prod_{\omega \in \act_{\prec}(T)}
\Bigg(\Big(\sum_{e \in \omega-\{T_{\omega}\}} \vect x_e\Big) + t{\vect x}_{T_{\omega}}\Bigg)\Bigg).\]
\end{theorem}

\begin{proof}
    By partitioning the set of transversals of $Z$ according to their $(Z,\prec)$-canonical compatible transversal, we obtain    \begin{align*}
        Q(Z;\vect x,t) &= \sum_{T'\in \mathcal T(Z)} t^{n(T')} \vect x_{T'}\\
        &= \sum_{T\in \Comp(Z,\prec)} \sum_{T':\ccl_{\prec}(T')=T} \Bigg(\prod_{\omega \in \inact_{\prec}(T)} \vect x_{T'_{\omega}}\Bigg)  \Bigg(\prod_{\omega \in \act_{\prec}(T)} \vect x_{T'_{\omega}}\Bigg) t^{n(T')} \\
        &= \sum_{T\in \Comp(Z,\prec)} \Bigg(\prod_{\omega \in \inact_{\prec}(T)} \vect x_{T_{\omega}}\Bigg) \sum_{T':\ccl_{\prec}(T')=T} \Bigg(\prod_{\omega \in \act_{\prec}(T)} \vect x_{T'_{\omega}}\Bigg) t^{n(T')},
        \end{align*}
    where the final equality follows from Proposition~\ref{prop:compatcriterion}. Moreover it also follows from Proposition~\ref{prop:compatcriterion} that the transversals contributing to the inner sum for a particular $\prec$-compatible transversal $T$ may include any member of each active skew class of $T$. Proposition~\ref{prop:nulltranscompat} tells us that the nullity of such a transversal $T'$ is equal to the number of skew classes which are active with respect to $T$ and on which $T$ and $T'$ agree. Thus
    \[ \sum_{T':\ccl_{\prec}(T')=T} \Bigg(\prod_{\omega \in \act_{\prec}(T)} \vect x_{T'_{\omega}}\Bigg) t^{n(T')} 
    =\Bigg(\prod_{\omega \in \act_{\prec}(T)}
\Bigg(\Big(\sum_{e \in \omega-\{T_{\omega}\}} \vect x_e\Big) + t{\vect x}_{T_{\omega}}\Bigg)\Bigg)\]
    and the result follows.
\end{proof}

Then we obtain the following corollaries. 
The first is immediate. The second comes from the previous theorem and by noting that Proposition~\ref{prop:nulltranscompat} implies that if $T$ is $\prec$-compatible then $n(T)=|{\act_{\prec}(T)}|$.

\begin{corollary}\label{cor:firstcompat}
Let $Z$ be a multimatroid and let $\prec$ be a total ordering of its skew classes. Then the unweighted transition polynomial is given by
\[ Q(Z;t)= \sum_{T\in \Comp(Z,\prec)}  
\Bigg(\prod_{\omega \in \act_{\prec}(T)}
 (|\omega|-1+t)\Bigg).\]
\end{corollary}

\begin{corollary}\label{cor:secondcompat}
Let $Z$ be a $q$-matroid and let $\prec$ be a total ordering of its skew classes. Then the unweighted transition polynomial is given by
\begin{equation} \label{eq:compsum} Q(Z;t)= \sum_{T\in \Comp(Z,\prec)}  
 (q-1+t)^{n(T)}.\end{equation}
\end{corollary}

\begin{example}[Examples~\ref{example},~\ref{eg:examplectd.1},~\ref{eg:examplectd.2}, and~\ref{eg:examplectd.3} continued]\label{eg:examplectd.4}
We again define a total ordering $\prec$ on the skew classes of $Z$, the $3$-matroid described in Example~\ref{example}, by requiring that $\omega_a\prec \omega_b\prec \omega_c$.
Recall from Example~\ref{eg:examplectd.3} that the $\prec$-compatible transversals of $Z$ are
$\{ \overline{a}, \overline{b}, \dott{c}\}$,
$\{ \overline{a}, \overline{b}, \overline{c}\}$,
$\{ \dott{a}, \dott{b}, \overline{c}\}$,
$\{ \widehat{a}, \widehat{b}, \overline{c}\}$,
$\{ \overline{a}, \overline{b}, \widehat{c}\}$,
$\{ \widehat{a}, \dott{b}, \widehat{c}\}$,
$\{ \dott{a}, \widehat{b}, \widehat{c}\}$.
We have $n(\{ \overline{a}, \overline{b}, \dott{c}\})=2$, but 
every other $\prec$-compatible transversal has nullity one. It follows from Corollary~\ref{cor:secondcompat} that
\[ Q(Z;t) = (2+t)^2 + 6(2+t) = t^2+10t+16,\]
which agrees with our earlier calculation in Example~\ref{eg:examplectd.1}.
\expleend\end{example}

By substituting $q=2$ and $t=0$ into Equation~\eqref{eq:compsum}, we observe that for a $2$-matroid $Z$, the evaluation $Q(Z;0)$ is equal to the number of $\prec$-compatible transversals of $Z$, for every ordering $\prec$ of the skew classes of $Z$. We noted earlier the number of bases of a non-degenerate multimatroid $Z$ is equal to $Q(Z;0)$, so every $2$-matroid has the same number of bases and $\prec$-compatible transversals, for every ordering $\prec$ of its skew classes. In the next subsection, we shall give a bijective proof of this fact and show that in the special case of a $2$-matroid $Z(M)$ coming from a matroid $M$, the bijection is essentially the same as that
introduced by 
Kochol~\cite{zbMATH07725327} and Pierson~\cite{zbMATH07814958}.

\begin{example}[Examples~\ref{eg:matroid},~\ref{eg:matcont},~\ref{eg:matcont1.5}, and~\ref{eg:matcont2} continued]\label{eg:matcont2.5}
We will deduce Theorem~\ref{thm:Tuttecompat} from Theorem~\ref{thm:cocomp}.
Let $M$ be a matroid with a total ordering $\prec$ of its element set $E$. Then, as we noted in Example~\ref{eg:matcont}, $\prec$ induces an ordering of the skew classes of $Z(M)$ which we also denote by $\prec$. We showed in Example~\ref{eg:matcont2} that a subset $A$ of $E(M)$ is in $\mathcal{D}(M,\prec)$ if and only if $\phi(A) \in \Comp(Z(M),\prec)$.
Suppose then that $A\in \mathcal{D}(M,\prec)$. As $\phi(A)\in \Comp(Z(M),\prec)$, a skew class $\omega_e$ is active with respect to $\phi(A)$ if and only if either $e \in A$ and there is a circuit $C$ of $M$ with $C\subseteq A$ and $e=\min(C)$, or $e\notin A$ and there is a cocircuit $C^*$ of $M$ with $C\subseteq E-A$ and $e=\min (C)$. After letting $\vect x_{\dott e}=u$ and $\vect x_{\overline e}=v$ and performing similar calculations to those in Equation~\eqref{eq:mat2}, the expression in Theorem~\ref{thm:cocomp} becomes
\[ Q(Z(M);\vect x,t) = \sum_{A\in \mathcal{D}(M,\prec)} 
u^{r_M(A)}v^{r_{M^*}(E-A)} (v+tu)^{|A|-r_M(A)} (u+tv)^{|E|-|A|-r_{M^*}(E-A)}.\] 
Using Equation~\eqref{eq:tuttetrans}, and substituting $u=\sqrt y$, $v=\sqrt x$ and $t=\sqrt{xy}$ gives
\begin{align*} \MoveEqLeft T(M;x+1,y+1)
\\&= \sqrt{x}^{\,r_M(E)-|E|} \sqrt{y}^{\,-r_M(E)}
\sum_{A\in \mathcal{D}(M,\prec)} 
\sqrt{y}^{\,r_M(A)}\sqrt{x}^{\,r_{M^*}(E-A)}\\
&\quad\quad \cdot (\sqrt{x}+y\sqrt{x})^{\,|A|-r_M(A)} (\sqrt{y}+x\sqrt{y})^{\,|E|-|A|-r_{M^*}(E-A)} \\
&= \sum_{A\in \mathcal{D}(M,\prec)}
\sqrt{x}^{\,r_M(E)-|E|+r_{M^*}(E-A)+|A|-r_M(A)}
\\
& \quad\quad \cdot
\sqrt{y}^{\,r_M(A) + |E|-|A|-r_{M^*}(E-A)-r_M(E)}\\
& \quad\quad\quad\quad  \cdot (x+1)^{|E|-|A|-r_{M^*}(E-A)}
(y+1)^{|A|-r_M(A)}\\
&= \sum_{A\in \mathcal{D}(M,\prec)} 
(x+1)^{r(E)-r(A)}(y+1)^{|A|-r(A)}.
\end{align*}
Thus we recover Theorem~\ref{thm:Tuttecompat}.
\expleend\end{example}

\subsection{An equivalence relation on transversals}\label{ss:results3}
The proof of Theorem~\ref{thm:cocomp} worked by partitioning the transversals of $Z$ according to their $\prec$-canonical compatible transversal. In the remainder of this section we consider the equivalence relation on the transversals for which this partition is the set of equivalence classes, in particular, its restriction to the set of bases.

Given a multimatroid $Z$ and a total ordering $\prec$ of its skew classes, we define an equivalence relation $\sim$ on $\mathcal{T}(Z)$
so that $T_1 \sim T_2$ if and only if $\ccl_{\prec}(T_1)=\ccl_{\prec}(T_2)$. It is clear that $\sim$ is genuinely an equivalence relation. Let $[T]$ denote the equivalence class of the transversal $T$ according to $\sim$. By Proposition~\ref{prop:compatcriterion} Part~\ref{prop:compatcriterion.1}
\[ [T] = \{ T'\in \mathcal T(Z) : T'_{\omega} = T_\omega, \ \omega\in \inact_{\prec}(\omega)\}.\]
Let $\mathcal T(Z)/{\sim}$ denote the set of equivalence classes of $\mathcal T(Z)$ according to $\sim$. It follows from Proposition~\ref{prop:propsofcompat} Part~\ref{prop:propsofcompat.3} that 
the image of $\ccl_{\prec}$ is $\Comp(Z,\prec)$, so
$|\mathcal T(Z)/{\sim}|=|{\Comp(Z,\prec)}|$.

In the case where $Z$ is non-degenerate, we now relate $\mathcal T(Z)/{\sim}$ to the collection $\mathcal H_{Z,\prec}$ from the previous section.

\begin{proposition}\label{prop:Handcompat}
Let $Z$ be a non-degenerate multimatroid with a total ordering $\prec$ of its skew classes. Then for every $(Z,\prec)$-compatible transversal $T$, we have
\[ \bigcup\limits_{\substack{B\in \mathcal B(Z):\\ \ccl_{\prec}(B)=T}}
H_{Z,\prec}(B) = \{T'\in \mathcal T(Z): \ccl_{\prec}(T')=T\}.\]
\end{proposition}
\begin{proof}
Members of $H_{Z,\prec}(B)$ differ only on active skew classes of $B$. Thus if $T'' \in H_{Z,\prec}(B)$, then by Proposition~\ref{prop:compatcriterion} Part~\ref{prop:propsofcompat.3} we have $\ccl(T'')=\ccl(B)$. This gives
\[ \bigcup\limits_{\substack{B\in \mathcal B(Z):\\ \ccl_{\prec}(B)=T}}
H_{Z,\prec}(B) \subseteq \{T'\in \mathcal T(Z): \ccl_{\prec}(T')=T\}.\]
To prove the reverse inclusion, observe that by Theorem~\ref{thm:maingen2}, for every transversal $T'$ with $\ccl_{\prec}(T')=T$, there is at least one basis $B$ for which $T'\in H_{Z,\prec}(B)$. 
As $B\in H_{Z,\prec}(B)$,
by the first part of the proof $\ccl_{\prec}(B)=\ccl_{\prec}(T')=T$ and the result follows.
\end{proof}

Thus $\mathcal T(Z)/{\sim}$ is closely related to $\mathcal H_{Z,\prec}$.
The advantage of locating transversals around $(Z,\prec)$-compatible transversals rather than bases, as we did in Theorem~\ref{thm:maingen2}, is that we always get a partition of $\mathcal T(Z)$.

Let $Z$ be a non-degenerate multimatroid and let $\prec$ be a total ordering of its skew classes. We now consider the restriction $\sim_{\mathcal B}$ of $\sim$ to $\mathcal B(Z)$. 
For a basis $B$, we define $[B]_{\mathcal{B}}:= \{ B'\in \mathcal B(Z): B'\sim_{\mathcal B} B\}$ and $\mathcal B(Z)/{\sim_{\mathcal{B}}} := \{ [B]_{\mathcal{B}} : B\in\mathcal B(Z)\}$. For each $(Z,\prec)$-compatible transversal $T$, 
we may choose $B$ such that $T\in H_{Z,\prec}(B)$. So Proposition~\ref{prop:Handcompat} implies that there is a basis $B$ such that $\ccl_{\prec}(B)=T$. Consequently $|\mathcal B(Z)/{\sim_{\mathcal B}}| = |\mathcal T(Z)/{\sim}| = |{\Comp(Z,\prec)}|$.

The preceding discussion together with Proposition~\ref{prop:compatcriterion} allows us to rewrite Corollary~\ref{cor:firstcompat} as a sum over the equivalence classes of ${\sim_{\mathcal B}}$.

\begin{corollary}
Let $Z$ be a non-degenerate multimatroid and let $\prec$ be a total ordering of its skew classes. Then the unweighted transition polynomial is given by
\[ Q(Z;t)= \sum_{[B]_{\mathcal B}\in \mathcal B(Z)/\sim_{\mathcal B} }  \: \Bigg( \prod_{\omega \in \act_{\prec}(B)}
 (t+|\omega|-1) \Bigg).\]
Here the sum is over equivalence classes, and the product is over the active skew classes of any equivalence class representative.
\end{corollary}

The next corollary is immediate.
\begin{corollary}\label{cor:equiv}
Let $Z$ be a $q$-matroid with $q\geq 2$ and let $\prec$ be a total ordering of its skew classes. Then the unweighted transition polynomial is given by
\begin{equation} Q(Z;t)= \sum_{[B]_{\mathcal B}\in \mathcal B(Z)/\sim_{\mathcal B}}   (t+q-1)^{|{\act_{\prec}(B)}|}.\label{eq:shwlling2}\end{equation}
Here the sum is over equivalence classes, and the exponent is the number of active skew classes of any equivalence class representative.
\end{corollary}

We now state our fourth main result which gives a succinct description of the members of each equivalence class of $\sim_{\mathcal B}$.
\begin{theorem}\label{thm:equiv}
Let $Z$ be a non-degenerate multimatroid
with a total ordering $\prec$ of its skew classes. Let $B$ be a basis of $Z$. Then a transversal $B'$ is in $[B]_{\mathcal B}$ if and only if $B'$ and $\ccl_{\prec}(B)$ agree on the skew classes which are inactive with respect to $B$ and disagree on the skew classes which are active with respect to $B$.
Moreover,
 \[
      |[B]_{\mathcal{B}}| = \prod_{\omega\in \act_{Z,\prec}(B)}(|\omega|-1).
  \]
\end{theorem}

\begin{proof} 
By Proposition~\ref{prop:compatcriterion} Part~\ref{prop:compatcriterion.1}, a transversal $B'$ satisfies $B'\sim B$ if and only if $B'$ and $B$ agree on the skew classes that are inactive with respect to $B$, which occurs if and only if $B'$ and $\ccl_{\prec}(B)$ agree on these skew classes. 

Now let $B'$ be such a transversal and let $T:=\ccl_{\prec}(B)$. To prove the first part it remains to prove that $B'$ is a basis if and only if   
$B'_{\omega}\ne T_{\omega}$
for every skew class $\omega$ that is active with respect to $B$. By definition, if $\omega$ is active with respect to $B$, we have $T_{\omega}=\underline B_{\omega}$.

Let $\omega$ be an active skew class with respect to $B$. Then, from the definition of activity, $Z$ has a circuit $C$ with $C \subseteq B \cup \{\underline B_{\omega}\}$ and $\min(C) = \underline B_{\omega} \in \omega$. Moreover, by Lemma~\ref{lem:compatcruc}, $Z$ has a circuit $C'$ with $C' \subseteq (B\cap B') \cup \omega $ and $\min (C') \in \omega$. 
By Lemma~\ref{lem:compatunique}, $\min(C')=\min(C)=\underline B_{\omega}$.
So we see that if $\underline B_\omega\in B'$ then $B'$ is dependent. 

Now suppose that $B'_{\omega} \ne \underline B_{\omega}$ for every skew class $\omega$ that is active with respect to $B$. Suppose for contradiction that $B'$ is not a basis. Then there must be a circuit $C''$ with $C''\subseteq B'$ and $C''_{\omega'} \ne B_{\omega'}$ for some skew class $\omega'$. By applying the argument used in the previous paragraph with $\omega=\omega'$, we see there is a circuit $C'$ with $C' \subseteq (B\cap B') \cup \omega' $ and $\min (C') = \underline B_{\omega'} \in \omega'$.
But then $C' \cup C''$ contains a unique skew pair $\{\underline B_{\omega'}, C''_{\omega'}\}$, contradicting Proposition~\ref{prop:uniqueskewpair}. So if $B'_{\omega}\ne \underline B_{\omega}$ for every skew class $\omega$ that is active with respect to $B$, then $B'$ is a basis, as required. 

The last part follows immediately.
\end{proof}

In the special case of a $2$-matroid, we get the following.
\begin{corollary}\label{cor:bij2m}
Let $Z$ be a $2$-matroid and let $\prec$ be a total ordering of its skew classes. Then 
every equivalence class of $\sim$ contains exactly one basis.

Moreover the map $f:\mathcal B(Z) \rightarrow \Comp(Z,\prec)$ 
with $B\mapsto \ccl_{\prec}(B)$ is a bijection between the collections of bases and $\prec$-compatible transversals of $Z$.
Explicitly, we have 
\[ f(B) = B \bigtriangleup
\bigcup_{\omega \in \act_{Z,\prec}(B)}  \omega.
\]
\end{corollary}

\begin{proof}
    The first assertion follows immediately by setting $|\omega|=2$ for every skew class $\omega$ in the previous theorem.

        As $|\mathcal B(Z)|= |{\Comp(Z,\prec)}|$, to see that the map $f$ is a bijection from $\mathcal B(Z)$ to $\Comp(Z,\prec)$, it is enough to note that for every basis $B$, $\ccl_{\prec}(B)\in \Comp(Z,\prec)$ by Proposition~\ref{prop:propsofcompat} Part~\ref{prop:propsofcompat.3}, and that $f$ is injective by the first part of the proof.

To see that the bijection $f$ is as described, first note that trivially every basis $B$ satisfies $B\in [B]_{\mathcal B}$. We deduce from the previous theorem that the skew classes on which $B$ and $\ccl_{\prec}(B)$ agree are precisely those which are inactive with respect to $B$. As $Z$ is a $2$-matroid, this is enough to determine $\ccl_{\prec}(B)$.
\end{proof}

\begin{example}  [Examples~\ref{eg:matroid},~\ref{eg:matcont},~\ref{eg:matcont1.5},~\ref{eg:matcont2}, and~\ref{eg:matcont2.5} continued]\label{eg:matcont3}
In this example we will show that the bijection $f$ from Corollary~\ref{cor:bij2m} is equivalent via the bijection $\phi$ to the bijection described by Kochol~\cite{zbMATH07725327} and Pierson~\cite{zbMATH07814958} between the bases of a matroid $M$ and the set $\mathcal D(M,\prec)$, where $\prec$ is a total ordering of the elements of $M$.

For a matroid $M$ with a total ordering $\prec$ of its elements, Kochol and Pierson describe a bijection $g:\mathcal B(M) \rightarrow \mathcal D(M,\prec)$ given by $g(B):=(B-\inta_{\prec}(B)) \cup \exta_{\prec}(B)$. In other words $g$ maps a basis $B$ to its antipodal point in the Boolean interval \mbox{$[ B-\inta_{\prec}(B), B\cup \exta_{\prec}(B)]$}.

We have
\[ \phi(g(B)) = \phi \Big(B\bigtriangleup \bigcup_{\substack{e\in \inta_{\prec}(B) \\ \phantom{e\in}\cup \exta_{\prec}(B)}} \omega_e \Big)
= f(\phi(B)). 
\]
\expleend\end{example}

We close this section by giving a combinatorial interpretation of the coefficients of $Q(Z;t-1)$.
Let $Z=(U,\Omega,r)$ be a $q$-matroid with $q\geq 2$, let $m:=|\Omega|$, and let $\prec$ be a total ordering of its skew classes. 
Then we can rewrite Equation~\eqref{eq:shwlling2} as
\[Q(Z;t)= \sum_{[B]_{\mathcal B}\in \mathcal B(Z)/\sim_{\mathcal B}}   ((t+1)+(q-2))^{|{\act_{\prec}(B)}|} = \sum_{i=0}^m a_i (t+1)^i,\]
for non-negative integers $a_0$, \ldots, $a_m$. By setting $t=0$, we see that  $\sum_{k=0}^{m} a_i=|\mathcal B(Z)|$.
We now give a combinatorial interpretation to the coefficients $a_i$.

For each equivalence class of $\sim_{\mathcal B}$, choose a representative and let $\hat B$ be the equivalence class representative of $B$. 
Finally, let 
$\mathcal{BR}(Z):=\{\hat B: B\in \mathcal B(Z)/{\sim}_{\mathcal B} \}$ be the set of all equivalence class representatives. 

It follows from Theorem~\ref{thm:equiv} that if $B_{\omega} \ne \hat B_{\omega}$ then there are $q-2$ possibilities for $B_{\omega}$. So we have
\begin{align*}
\sum_{[B]_{\mathcal B}\in \mathcal B(Z)/\sim_{\mathcal B}}   ((t+1)+(q-2))^{|{\act_{\prec}(B)}|}
&= 
\sum_{\hat B\in \mathcal{BR}(Z)} 
\sum_{B\in [\hat B]_{\mathcal B}} (t+1)^{|{\act_{\prec}(B)}| - |B\bigtriangleup \hat B|} \\
&= 
\sum_{B\in \mathcal{B}(Z)} (t+1)^{|{\act_{\prec}(B)}| - |B\bigtriangleup \hat B|}.
\end{align*}
Hence
\begin{align*} a_i &= |\{ \mathcal B\in \mathcal B(Z): |{\act_{\prec}(B)}| - |B\bigtriangleup \hat B| = i| \}|
\\&= |\{ \mathcal B\in \mathcal B(Z): 
|B\cap \hat B| = |\Omega|-|{\act_{\prec}(B)}|+i = |{\inact_{\prec}(B)}|+i\}|.\end{align*}

\begin{example}
[Examples~\ref{example},~\ref{eg:examplectd.1},~\ref{eg:examplectd.2},~\ref{eg:examplectd.3}, and~\ref{eg:examplectd.4} continued]
In $Z$, the 3-matroid from Example~\ref{example}, with the total ordering $\prec$ on its skew classes, the equivalence class of $\{\dott a,\dott b,\dott c\}$ also contains $\{\widehat{a},\dott b,\dott c\}$, $\{\dott a,\widehat{b},\dott c\}$ and $\{\widehat{a},\widehat{b},\dott c\}$.
 It contributes one to $a_0$, two to $a_1$ and one to $a_2$.
 Each of the other six equivalence classes contains two bases, and contributes one to both $a_0$ and $a_1$. Thus $a_2=1$, $a_1=8$ and $a_0=7$. We have $a_3=0$, because $|{\inact_{\prec}(B)}|>0$ for every basis $B$. Hence
\[
  Q(Z;t) = (t+1)^2+8(t+1)+7= t^2+10t+16.
\]

Alternatively, one could obtain this expression directly from Corollary~\ref{cor:equiv}. As $\sim_{\mathcal B}$ has one equivalence class whose members all have two active skew classes, and six equivalence classes whose members all have one active skew class, we get
\[  Q(Z;t) = (t+2)^2+6(t+2)= t^2+10t+16.
\]
\expleend\end{example}

\section{Delta-matroids}\label{sec:dm}

In this section we explore the implications of Theorem~\ref{thm:main1} and Corollary~\ref{cor:main2} for delta-matroids and connect our results with those of Morse~\cite{zbMATH07067836}. 
Example~\ref{eg:lasteg} (below) provides a running example for the constructions in this section. 
Delta-matroids and equivalent structures were introduced in the mid 1980s by several authors~\cite{MR904585,zbMATH04070920,zbMATH03985246}, but principally studied by Bouchet.
A delta-matroid is a pair $(E,\mathcal F)$ where $E$ is a finite set and $\mathcal F$ is a non-empty collection of subsets of $E$ satisfying the \emph{symmetric exchange axiom}:
\begin{quote}
  for all triples $(X,Y,u)$ with $X$ and $Y$ in $\mathcal{F}$
  and $u\in X\triangle Y $, there is a $v \in X\triangle Y$ (perhaps
  $u$ itself) such that $X\triangle \{u,v\}$ is in $\mathcal{F}$.
\end{quote}
The elements of $\mathcal F$ are called \emph{feasible sets}.
A matroid (defined in terms of its bases) is precisely a delta-matroid in which all the feasible sets have the same size.

Bouchet established the relationship between $2$-matroids and delta-matroids in~\cite{MR904585}. Describing and using this relationship places a strain on the notation, so to aid the exposition, we say that a $2$-matroid with carrier $(U,\Omega)$ is \emph{indexed} by a finite set $E$ if $\Omega=\{\omega_e:e \in E\}$.
Now suppose that we have a $2$-matroid with carrier $(U,\Omega)$ indexed by $E$. Fix a transversal $T$ of $\Omega$. For a subtransversal $S$ of $\Omega$, we let 
\[ T(S) = \{e \in E: T \cap \omega_e = S\cap \omega_e\}.\]
\begin{theorem}[Bouchet~\cite{MR904585}]\label{thm:equivalence}
Let $Z=(U,\Omega,r)$ be a $2$-matroid indexed by $E$ and let $T$ be a transversal of $\Omega$. Let $\mathcal F := \{T(B):B\in \mathcal B(Z)\}$.
Then $(E,\mathcal F)$ is a delta-matroid.
Moreover, up to isomorphism, every delta-matroid arises in this way.
\end{theorem}
We denote the delta-matroid described in the theorem by $D(Z,T)$. 
The process is reversible and every $2$-matroid may be constructed from a delta-matroid by reversing the process. 
Given a delta-matroid $D=(E,\mathcal F)$, let $E_1$ and $E_2$ denote two disjoint copies of $E$. Now let $\mu_1$ and $\mu_2$ be bijections from $E$ to $E_1$ and $E_2$, respectively. Then we obtain a $2$-matroid from $D$ indexed by $E$ with carrier $(E_1\cup E_2, \{\omega_e: e\in E\})$, so that
for all $e$ we have $\omega_e=\{ \mu_1(e), \mu_2(e)\}$,
and collection of bases $\{\mu_1(F) \cup \mu_2(E-F) : F \in \mathcal F\}$. 
Given a delta-matroid $D=(E,\mathcal F)$ we use $Z(D,\mu_1,\mu_2)$ to denote this $2$-matroid. We have $D=D(Z(D,\mu_1,\mu_2),\mu_1(E))$.
As every matroid is a delta-matroid,
the construction we have just described extends the one given in Example~\ref{eg:matroid} where we described how to associate a $2$-matroid with a matroid.

As a consequence, definitions and results for $2$-matroids may be translated to delta-matroids and vice versa. We now effect this translation for some of the concepts we have introduced. 
In order to simplify the notation, when moving from a delta-matroid $(E,\mathcal F)$ to a $2$-matroid, we will take $E_1:=\dott E:=\{\dott e:e\in E\}$, $E_2:=\overline E:=\{\overline e:e\in E\}$, $\mu_1(e)=\dott e$ and $\mu_2(e)=\overline e$. We write $Z(D)$ for $Z(D,\mu_1,\mu_2)$. 
In order to translate between subsets of $E$ (thought of in the delta-matroid $D$) and the corresponding transversal of $Z(D)$, we extend the scope of the map $\phi$ from Example~\ref{eg:matroid} to delta-matroids.
For a subset $A$ of $E$ we let $\phi(A):=\{\dott e : e\in A\}\cup \{\overline e: e\not\in A\}$.

Following~\cite{MR3191496}, for a delta-matroid $D=(E,\mathcal F)$ 
and subset $X$ of $E$, we define 
\[ d(X) := \min_{F\in \mathcal F} |F\bigtriangleup X|.\]
Let $Z=(U,\Omega,r)$ be a multimatroid.
Recall that for a transversal $T$ of $\Omega$, we have $r_Z(T)=\max\{|T\cap B|:B\in \mathcal B(Z)\}$. 
From this it is easy to deduce that $d(X)=n_{Z(D)}(\phi(X))$.

We define the \emph{transition polynomial} of a delta-matroid $D=(E,\mathcal F)$ by
\[Q(D;w,x,t) := \sum_{A \subseteq E} w^{|E-A|}x^{|A|} t^{d(A)}.\]
Observe that the transition polynomial of $D$ can be obtained from the weighted transition polynomial of $Z(D)$ by letting 
$\vect x_{\dott e}=x$ and $\vect x_{\overline e}=w$ for all $e$ in $E$.
This is the polynomial $Q_{w,x,0}(D;t)$ of~\cite{MR3191496}.

Following~\cite{zbMATH07067836}, for a feasible set $F$ and element $e$ of a delta-matroid $D$, we say that $e$ is $F$-\emph{orientable} if $F\bigtriangleup \{e\}$ is not feasible. So $e$ is $F$-orientable if and only if the fundamental circuit $C(\phi(F),\omega_e)$ exists in $Z(D)$. Let $e$ and $f$ be distinct elements of $D$ such that $e$ is $F$-orientable. Then, following~\cite{zbMATH07067836}, we say that $e$ is $F$-\emph{interlaced} with $f$ if $F\bigtriangleup \{e,f\}$ is feasible. (Notice that for $e$ to be $F$-interlaced with $f$, we require $e$ to be $F$-orientable, but there is no such requirement on $f$. In~\cite{zbMATH07067836}, interlacement is described more generally, but we will not need that.)

The next lemma is closely related to Proposition~6.1 of~\cite{MM3zbMATH01648954}.
Its translation to delta-matroids will be useful here and it will itself be useful in the final section.
\begin{lemma}\label{lem:blah}
Let $Z$ be a non-degenerate multimatroid. Suppose that $B$ is a basis of $Z$ and $\omega$ is a skew class such that $C(B,\omega)$ exists. Then an element $f$ of $B-\omega$ is in $C(B,\omega)$ if and only if the skew class containing $f$ has an element $x$ such that 
    $B\bigtriangleup\{B_{\omega}, \underline B_{\omega}, f, x\}$ is a basis of $Z$.
 \end{lemma}

 \begin{proof}
Let $f$ be in $B-\omega$ and let $\omega'$ be the skew class containing $f$. Let $B':=B \bigtriangleup \{B_{\omega}, \underline B_{\omega}\}$. 
Suppose first that $f$ is in $C(B,\omega)$. Then $B'$ contains a unique circuit which includes $f$. So  $B'-\{f\}$ is independent but $B'$ is not. Hence by~\ref{axiom2} for every element $x$ of $\omega'-\{f\}$, $B'\bigtriangleup\{f,x\}$ is a basis of $Z$.

Now suppose that $\omega'$ has an element $x$ such that $B'\bigtriangleup\{f,x\}$ is a basis of $Z$. Then $C(B,\omega)\subseteq B'$, but $B'-\{f\}$ does not contain a circuit. Hence $f\in C(B,\omega)$.
 \end{proof}

\begin{corollary}\label{lem:activities}
Let $D=(E,\mathcal F)$ be a delta-matroid and let $Z:=Z(D)$.
Let $F$ be a feasible set of $D$, and $e$ and $f$ be distinct elements of $D$ so that $e$ is $F$-orientable. Then $e$ is $F$-interlaced with $f$ if and only if 
$C(\phi(F),\omega_e)\cap \omega_f\ne \emptyset$. 
\end{corollary}
\begin{proof}
As $e$ is $F$-orientable, $C(\phi(F),\omega_e)$ exists, so we may apply Lemma~\ref{lem:blah}. 
This implies that $C(\phi(F),\omega_e)$ meets $\omega_f$ 
if and only if $\phi(F) \bigtriangleup \omega_e \bigtriangleup \omega_f$ is a basis of $Z(D)$ which occurs if and only if $F\bigtriangleup \{e,f\}$ is feasible in $D$.
\end{proof}

Now suppose that we have a total ordering $\prec$ of the elements of a delta-matroid $D$. Note that there is a corresponding total ordering of the skew classes of $Z(D)$, which we also denote by $\prec$.
Again, following~\cite{zbMATH07067836}, we say that an element $e$ of $D$ is \emph{active} with respect to a feasible set $F$ of $D$ if $e$ is $F$-orientable and there is no element $f$ with $f\prec e$ such that $e$ is $F$-interlaced with $f$. The next result follows immediately from Lemma~\ref{lem:activities} and completes the translation of $2$-matroid concepts to delta-matroids. 
\begin{corollary}
Let $D$ be a delta-matroid with a total ordering $\prec$ of its skew classes. Let $F$ be a feasible set of $D$. Then element $e$ is active with respect to $F$ in $D$ if and only if the skew class $\omega_e$ is active with respect to $\phi(F)$ in $Z(D)$.
\end{corollary}
If $e$ is active with respect to $F$ and belongs (does not belong) to $F$, then we say that $e$ is \emph{internally active} (\emph{externally active}) with respect to $F$. 
Notice that $e$ is internally active with respect to $F$ 
if and only if $\omega_e$ is active with respect to $\phi(F)$ in $Z(D)$ and $\dott e \in \phi(F)$; similarly, $e$ is externally active with respect to $F$ 
if and only if $\omega_e$ is active with respect to $\phi(F)$ in $Z(D)$ and $\overline e \in \phi(F)$.
We denote the set of internally (externally) active elements with respect to $F$ by $\inta_{\prec}(F)$ ($\exta_{\prec}(F)$).

Having translated the necessary $2$-matroid concepts to delta-matroids, we can obtain Theorem~4.15 of~\cite{zbMATH07067836} as a special case of Theorem~\ref{thm:main1}.
\begin{theorem}\label{thm:delta1}
Let $D=(E,\mathcal F)$ be a delta-matroid and let $\prec$ be a total ordering of $E$. Then
\[ Q(D;w,x,t) =\sum_{F\in \mathcal F} w^{|E|-|F|}x^{|F|} (1+(w/x)t)^{|{\inta_{\prec}(F)}|} (1+(x/w)t)^{|{\exta_{\prec}(F)}|}.\]
\end{theorem}
\begin{proof}
Let $Z:=Z(D)$. 
For each $e$ in $E$, let $\vect x_{\dott e}=x$ and $\vect x_{\overline e}=w$.
Let $F$ be feasible in $D$ and let $B:=\phi(F)$. 
Then $|{\inact_{\prec}(B) \cap \dott E}| = |F|-|{\inta_{\prec}(F)}|$
and $|{\inact_{\prec}(B) \cap \overline{E}}| = |E-F|-|{\exta_{\prec}(F)}|$. So we have
\begin{align*}
Q(D;w,x,t)&=Q(Z;\vect x,t)\\ &= \sum_{B\in \mathcal B(Z)} \Bigg(\prod_{\omega \in \inact_{\prec}(B)}\vect x_{B_{\omega}}\Bigg) \Bigg(\prod_{\omega \in \act_{\prec}(B)}
 \Big(\frac {t\vect x_{\underline B_{\omega}}} {|\omega|-1}+\vect x_{B_{\omega}}\Big)\Bigg)\\
&= \sum_{F\in\mathcal F} x^{|{F}|-|{\inta_{\prec}(F)}|} w^{|{E}|-|{F}|-|{\exta_{\prec}(F)}|} (tw+x)^{|{\inta_{\prec}(F)}|}(tx+w)^{|{\exta_{\prec}(F)}|}\\
&= \sum_{F\in\mathcal F} x^{|{F}|} w^{|{E}|-|{F}|} (tw/x+1)^{|{\inta_{\prec}(F)}|}(tx/w+1)^{|{\exta_{\prec}(F)}|}.
\end{align*}
\end{proof}

Our final result of this section is the special case of Corollary~\ref{cor:main2} for delta-matroids. Although this is not explicitly stated in~\cite{zbMATH07067836}, it follows easily by combining Lemmas~4.7, 4.8 and the proof of Corollary~4.14 from the same reference.
Let $D=(E,\mathcal F)$ be a delta-matroid and let $\prec$ be a total ordering of the elements of $E$.
For any feasible set $F$, let $I(F)$ be the Boolean interval given by $I(F):=[F-\inta_{\prec}(F), F\cup \exta_{\prec}(F)]$.

\begin{theorem}\label{thm:delta2}
Let $D$ be a delta-matroid with element set $E$
and let $\prec$ be a total ordering of $E$. Then the
collection of Boolean intervals $\{ I(F): F \in \mathcal F\}$ forms a partition of $2^{E}$.
\end{theorem}

Every matroid is a delta-matroid and, as pointed out in~\cite[Theorem~3.8]{zbMATH07067836}, the notions of internal and external activities for matroids and delta-matroids are consistent. 
Thus Theorem~\ref{thm:partition} is a special case of Theorem~\ref{thm:delta2}.

\begin{example}\label{eg:lasteg}
Consider the 2-matroid $Z'$ with skew classes $\omega_a=\{\dott a,\overline{a}\}$, $\omega_b=\{\dott b,\overline{b}\}$
and $\omega_c=\{\dott c,\overline{c}\}$, having bases $\{\dott a,\dott b,\dott c\}$,
$\{\dott a, \overline{b}, \overline{c}\}$ and
$\{ \overline{a}, \dott b, \overline{c}\}$.
Thus $Z'$ is indexed by $\{a,b,c\}$. We define a total ordering $\prec$ on the skew classes by requiring that $\omega_a\prec \omega_b\prec \omega_c$.
The circuits of $Z'$ are the subtransversals
$\{ \overline{a}, \overline{b}\}$, $\{ \overline{a}, \dott c\}$, $\{ \overline{b}, \dott c\}$ and
$\{ \dott a, \dott b, \overline{c}\}$. Thus, the basis $\{\dott a, \dott b, \dott c\}$ has $\omega_a$ and $\omega_b$ as active skew classes while
$\{ \dott a, \overline{b}, \overline{c}\}$ and $\{ \overline{a}, \dott b, \overline{c}\}$ have
$\omega_a$ as an active skew class.

The delta-matroid $D:=D(Z',\{\dott a, \dott b, \dott c\})$ has feasible sets
$\{a,b,c\}$, $\{a\}$ and $\{b\}$. 
Moreover, $D$ inherits from $Z'$ the total ordering $\prec$ on its elements given by $a\prec b \prec c$.
By Lemma~\ref{lem:activities}, $a$ and $b$ are internally active in $\{a,b,c\}$ while $a$ is  internally active in  $\{a\}$ and externally active in $\{b\}$. Then,
\begin{align*}
Q(D;w,x,t) &= x^{3} (1+(w/x)t)^{2} +w^{2}x (1+(w/x)t) +w^{2}x  (1+(x/w)t)\\
 &= x^3+3wx^2t+w^2xt^2+2w^2x+w^3t.
\end{align*}

In the 2-matroid $Z'$,
\begin{gather*} H(\{\dott a,\dott  b,\dott  c\}) = \{ \{\dott a, \dott b, \dott c\}, \{\overline{a}, \dott b, \dott c\}, \{\dott a, \overline{b}, \dott c\}, \{\overline{a}, \overline{b}, \dott c\}\},\\
H(\{\dott  a, \overline{b}, \overline{c}\})=\{\{ \dott a, \overline{b}, \overline{c}\},\{ \overline{a}, \overline{b}, \overline{c}\}\} \text{ and } H(\{ \overline{a}, \dott b, \overline{c}\}) = \{\{ \overline{a}, \dott b, \overline{c}\},\{\dott  a,\dott  b, \overline{c}\}\}.\end{gather*}
Thus, in the delta-matroid $D$,
\begin{gather*} I(\{a,b,c\})=\{\{a,b,c\},\{b,c\},\{a,c\},\{c\}\},\\ I(\{a\}) = \{\{a\},\emptyset\} \text{ and } I(\{b\})= \{\{a,b\},\{b\}\}.\end{gather*}
\expleend\end{example}

\section{Ribbon graphs and the topological transition polynomial}\label{sec:newtopo}

In this section we apply our results to topological graph polynomials. The \emph{topological transition polynomial},  introduced in \cite{MR2869185}, is a multivariate polynomial of ribbon graphs (or equivalently, of graphs embedded in surfaces). It contains both the 2-variable version of Bollob\'as and Riordan's ribbon graph polynomial~\cite{MR1851080,MR1906909} and the Penrose polynomial~\cite{MR1428870,MR2994409} as specializations, and is intimately related to  Jaeger's transition polynomial~\cite{MR1096990} and the 
 generalized transition polynomial of 
 \cite{MR1980048}.
Describing the polynomial requires some background on ribbon graphs.  We keep our exposition brief and refer the reader to~\cite{MR3086663,zbMATH07553843} for additional details.

A {\em ribbon graph} $\bG=\left(V,E\right)$ is a surface with boundary, represented as the union of two sets of discs --- a set $V$ of {\em vertices} and a set $E$ of {\em edges} --- such that: (1) the vertices and edges intersect in disjoint line segments; (2) each such line segment lies on the boundary of precisely one vertex and precisely one edge; and (3) every edge contains exactly two such line segments. 
As every ribbon graph $\bG$ can be regarded as a surface with boundary, it has some number of boundary components, which we shall denote by $b(\bG)$. The boundary components will be particularly important in what follows. 
A ribbon graph is said to be a \emph{quasi-tree} if it has exactly one boundary component. We let $k(\bG)$ denote the number of connected components of $\bG$. 
As an example, Figure~\ref{newfex1} shows a ribbon graph with two vertices, three edges and one boundary component. It is a quasi-tree.

Let $\bG=(V,E)$ be a ribbon graph and $e\in E$. Then $\bG\ba e$ denotes the ribbon graph obtained from $\bG$ by \emph{deleting} the edge $e$. For $A\subseteq E$, $\bG\ba A$ is the result of deleting each edge in $A$ (in any order). Ribbon graphs of the form  $\bG\ba A$ for some $A\subseteq E$ are the  \emph{spanning ribbon subgraphs} of $\bG$. 
We use $\bG|_A$ to denote $\bG\ba (E-A)$.

The \emph{partial Petrial} with respect to an edge of a ribbon graph $e$,  introduced in \cite{MR2869185} and denoted here by $\bG^{\tau(e)}$, is, informally, the result  of detaching one end of the edge $e$ from its incident vertex, giving the edge a half-twist, and reattaching it. Formally, it is obtained by  detaching an end of $e$ from its incident vertex $v$ creating arcs $[a,b]$ on $v$, and $[a',b']$ on $e$  (so that $\bG$ is recovered by identifying $[a,b]$ with $[a',b']$), then reattaching the end  by identifying the arcs antipodally (so that the arc $[a,b]$ is identified with the arc $[b',a']$). 
 For $A\subseteq E(G)$,  the partial Petrial $\bG^{\tau(A)}$ is the result of forming the partial Petrial with respect to every element in $A$ (in any order).  
 There is a natural correspondence between the edges of $\bG$ and  $\bG^{\tau(A)}$, and so we can and will assume that they have the same edge set.
 We note that $\bG^{\tau(E)}$ is the Petrie dual, defined in~\cite{MR547621}, of $\bG$.
 Figure~\ref{newfex} shows a ribbon graph and one of its partial Petrials.

\medskip

Let $\bG=(V,E)$ be a ribbon graph, $\boldsymbol\alpha=\{\alpha_e\}_{e\in E}$, $\boldsymbol\beta=\{\beta_e\}_{e\in E}$, $\boldsymbol\gamma=\{\gamma_e\}_{e\in E}$ be indexed families  of indeterminates, and $t$  be another indeterminate. Then the
 \emph{topological transition polynomial}, $Q(\bG; (\boldsymbol\alpha, \boldsymbol\beta, \boldsymbol\gamma) , t)$, can be defined as 
 \[Q(\bG; (\boldsymbol\alpha, \boldsymbol\beta, \boldsymbol\gamma) , t) :=\sum_{(X,Y,Z) \in \mathcal{P}_3(E)}    \Big( \prod_{e\in X}\alpha_e\Big) \Big(  \prod_{e\in Y}  \beta_e\Big) \Big( \prod_{e\in Z}   \gamma_e\Big)  t^{b( \bG^{\tau(Z)}\ba Y)},\]
where $\mathcal{P}_3(E)$ denotes the set of ordered partitions of $E$ into three blocks, which may be empty, and where $b(\bG^{\tau(Z)}\ba Y)$ denotes the number of boundary components of $\bG^{\tau(Z)}\ba Y$.

\medskip

We shall apply Theorem~\ref{thm:main1} to obtain an activities expansion for the topological transition polynomial. This expansion is analogous to  the activities expansion for the Tutte polynomial (Theorem~\ref{thm:tutteact}), and the quasi-tree expansions for the Bollob\'as--Riordan and Krushkal polynomials~\cite{But,MR2854567,CPC,MR2780852}. 

\medskip

For a  ribbon graph $\bG=(V,E)$ we set 
\[\mathcal{Q}(\bG) := \{  (X,Y,Z) \in \mathcal{P}_3(E) :  b(\bG^{\tau(Z)} \ba Y)=k(\bG)  \}.\]
Thus when $(X,Y,Z)\in \mathcal{Q}(\bG)$, the ribbon subgraph $\bG^{\tau(Z)} \ba Y$ comprises a spanning quasi-tree of each component of $\bG^{\tau(Z)}$.

Now let $(X,Y,Z)$ be in $\mathcal{Q}(\bG)$  and consider the spanning ribbon subgraph $\bG^{\tau(Z)}  \!\ba\! Y$ as a subset of the ribbon graph  $\bG^{\tau(Z)}$. 
For each connected component of $\bG$ the corresponding unique boundary component  of $\bG^{\tau(Z)} \ba Y$ defines a curve $\partial$ in $\bG^{\tau(Z)}$ that meets each edge $e$ of the connected component in exactly two arcs $e_1$ and $e_2$. 
If we arbitrarily orient each edge then the arcs $e_1$ and $e_2$ are directed. 
We use these directed arcs to classify the edges of $\bG$ as follows.
Arbitrarily orient the curve $\partial$. 
If the direction of one arc $e_1$ or $e_2$ agrees with  the direction of $\partial$  and the other disagrees then we say that $e$ is \emph{$(X,Y,Z)$-nonorientable}, otherwise it is \emph{$(X,Y,Z)$-orientable}.
Next, if when travelling around $\partial$ two edges $e$ and $f$ of $\bG$ are met in the cyclic order $efef$ then we say that $e$ and $f$ are \emph{$(X,Y,Z)$-interlaced}.

Now suppose that $\bG$ has  a total ordering $\prec$  of its edge set. 
Given $(X,Y,Z)\in \mathcal{Q}(\bG)$ we say that an edge $e$ of $\bG$ is \emph{active} with respect to $(X,Y,Z)$ if it is \emph{not} $(X,Y,Z)$-interlaced  with an edge of lower order, otherwise it is \emph{inactive} with respect to $(X,Y,Z)$. 
An edge $e$ is \emph{active-orientable} with respect to $(X,Y,Z)$ if it is both active with respect to $(X,Y,Z)$ and $(X,Y,Z)$-orientable.
The terms \emph{active-nonorientable}, \emph{inactive-orientable}, and \emph{inactive-nonorientable} are defined similarly.
Furthermore for a subset $S$ of $E$, we let $\IA(S)$, $\AO(S)$ and $\AN(S)$ denote the sets of edges in $S$ which are inactive with respect to $(X,Y,Z)$, active-orientable  with respect to $(X,Y,Z)$, and active-nonorientable with respect to $(X,Y,Z)$, respectively. 
(Although missing from the notation, $(X,Y,Z)$ will always be clear from context.)

\begin{figure}[!t]
    \centering
         \subfloat[$\bG$.]{
  \labellist
\small\hair 2pt
\pinlabel {$a$} at    75 72
\pinlabel {$b$} at     75 26
\pinlabel {$c$} at     11 72 
\endlabellist
 \includegraphics[scale=1]{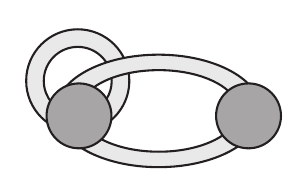} 
        \label{newfex1}}
        \quad
 \subfloat[$\bG^{\tau(a)}$]{
  \labellist
\small\hair 2pt
\pinlabel {$a$} at    75 72
\pinlabel {$b$} at     90 27
\pinlabel {$c$} at     11 72  
\endlabellist
 \includegraphics[scale=1]{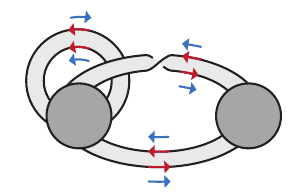} 
        \label{newfex2}}       
     \caption{Examples of ribbon graphs.}
         \label{newfex}
\end{figure}
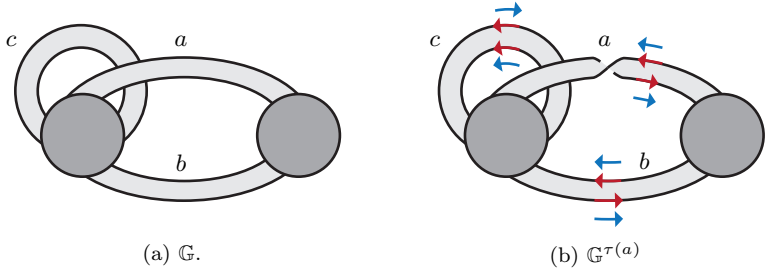

\begin{theorem}\label{newtt1}
Let $\bG$ be a  ribbon graph  and let 
$\prec$ be a total ordering of its edge set. Then the topological transition polynomial is given by
\begin{multline}\label{thm:newttpexp} Q(\bG; (\boldsymbol\alpha, \boldsymbol\beta, \boldsymbol\gamma) , t)  =    
t^{k(\bG)} \cdot 
\sum_{(X,Y,Z)\in\mathcal{Q}(\bG)}
\:\:
 \prod_{e\in\IA(X)}\alpha_e
 \prod_{e\in\IA(Y)}\beta_e
 \prod_{e\in\IA(Z)}\gamma_e
\\
\prod_{e\in\AO(X)}\left( \frac{t\beta_e}{2}+\alpha_e  \right)
 \prod_{e\in\AO(Y)}\left( \frac{t\alpha_e}{2}+\beta_e  \right)
\prod_{e\in\AO(Z)}\left( \frac{t\beta_e}{2}+\gamma_e  \right)
\\
\prod_{e\in\AN(X)}\left( \frac{t\gamma_e}{2}+\alpha_e  \right)
\prod_{e\in\AN(Y)}\left( \frac{t\gamma_e}{2}+\beta_e  \right)
\prod_{e\in\AN(Z)}\left( \frac{t\alpha_e}{2}+\gamma_e  \right).
\end{multline}
\end{theorem}

\begin{example}
Let $\bG$ be the ribbon graph shown in Figure~\ref{newfex1}  and let $a\prec b \prec c$. 
Then $(\{b,c\},\emptyset, \{a\}) \in \mathcal{Q}(\bG)$. Figure~\ref{newfex2} shows $ \bG^{\tau(a)}\ba \emptyset $. 
In the figure the blue arrows (drawn off the ribbon graph) indicate the directed arcs arising from an arbitrary edge orientation and the red arrows (drawn on the ribbon graph) indicate an orientation of $\partial$.
We see that $a$ and $c$ are $(\{b,c\},\emptyset, \{a\})$-interlaced, as are $b$ and $c$, but $a$ and $b$ are not. Furthermore, edges $a$ and $b$ are $(\{b,c\},\emptyset, \{a\})$-orientable while $c$ is $(\{b,c\},\emptyset, \{a\})$-nonorientable. Thus we see that $a$ and $b$ are active-orientable  and $c$ is inactive with respect to $(\{b,c\},\emptyset, \{a\})$, resulting in the summand
$\alpha_c(t\beta_a/2+\gamma_a)(t\beta_b/2+\alpha_b)$. 
\expleend\end{example}

We remark that in Equation~\eqref{thm:newttpexp}, the reader may expect all of $\frac{t\alpha_e}{2}$, $\frac{t\beta_e}{2}$, and $\frac{t\gamma_e}{2}$ to appear in both the second and third lines.
The  apparent asymmetry in the expansion results from our use of the standard definition of $(X,Y,Z)$-orientable. It is possible to use an alternative notation to eliminate the asymmetry.

We shall deduce Theorem~\ref{newtt1} from Theorem~\ref{thm:main1} by describing the construction of a $3$-matroid from a ribbon graph.   
Before doing so, we recall a few facts concerning delta-matroids associated with ribbon graphs and a construction due to Brijder and Hoogeboom which associates a uniquely defined $3$-matroid with each delta-matroid from a certain subclass. Let $\bG=(V,E)$ be a ribbon graph. Bouchet proved in~\cite{zbMATH04185622} that the pair $(E,\mathcal F)$, where 
\[ \mathcal F= \{A\subseteq E: b(\bG |_A) = k(\bG)\},\]
is a delta-matroid. We denote this delta-matroid by $D(\bG)$. 
A delta-matroid arising in this way is said to be \emph{ribbon-graphic}.
(Note that Bouchet's proof uses the language of combinatorial maps rather than ribbon graphs. For a proof in the ribbon graph setting see~\cite{CMNR1}.)

Brijder and Hoogeboom~\cite{zbMATH05982480} defined a class of delta-matroids they called \emph{vf-safe} delta-matroids. The definition of vf-safe is not needed here, but by combining results of Bouchet~\cite{zbMATH04162893,zbMATH04185622} and Brijder and Hoogeboom~\cite[Theorem~8.2]{zbMATH06181926} it may be shown that every ribbon-graphic delta-matroid is vf-safe. (See also~\cite[Theorem~5.44]{CMNR1}). 
Brijder and Hoogeboom~\cite{zbMATH05982480} define the operation of \emph{loop-complementation} with respect to a subset of the elements of a delta-matroid  (again we omit the definition as it is not needed here) and show that for vf-safe delta-matroids 
loop complementation always preserves vf-safeness. Given a delta-matroid $D$ and subset $A$ of its elements, we denote the loop-complementation with respect to $A$ by $D+A$. In~\cite[Theorem~4.1]{CMNR2}, it is shown that if $\bG$ is a ribbon graph, then $D(\bG^{\tau(A)})=D(\bG)+A$.

Given a vf-safe delta-matroid $D=(E,\Omega)$, Brijder and Hoogeboom~\cite{MR3191496} describe the construction of a $3$-matroid $Z_3(D)$.
Let $E_1$, $E_2$ and $E_3$ be disjoint copies of $E$, and let $\mu_1$, $\mu_2$ and $\mu_3$ be bijections from $E$ to $E_1$, $E_2$ and $E_3$. Then we obtain a $3$-matroid from $D$ indexed by $E$ with carrier $(E_1\cup E_2\cup E_3, \{\omega_e: e\in E\})$, so that
for all $e$ we have $\omega_e=\{ \mu_1(e), \mu_2(e), \mu_3(e)\}$. 
For each $(X,Y,Z)$ in $\mathcal P_3(E)$, we let 
\[T(X,Y,Z):= \mu_1(X) \cup \mu_2(Y) \cup \mu_3(Z).\]
The transversal $T(X,Y,Z)$ is a basis of $Z_3(D)$ if and only $X\cup Z$ is feasible in $D+Z$. Brijder and Hoogeboom showed in~\cite[Theorem~16]{MR3191496} that $Z_3(D)$ is genuinely a multimatroid.

The $3$-matroid $Z_3(D)$ has an important additional property. Following Bouchet \cite{MM3zbMATH01648954}, we say that a multimatroid $Z=(U,\Omega,r)$ is \emph{tight} if it is non-degenerate and for every subtransversal $S$ with $|S|=|\Omega|-1$,
 there is an element $x$ of the skew class not meeting $S$ such that $r(S\cup\{x\})=r(S)$. Equivalently,~\cite[{Theorem~4.2}]{MM3zbMATH01648954}, a multimatroid is tight if and only if it is non-degenerate and the union of a basis and a skew class always contains a circuit.
Brijder and Hoogeboom~\cite{MR3191496} show that the $Z_3(D)$ is always tight.

We now describe how a $3$-matroid $(U(\bG), E, r)$  can be obtained from a  ribbon graph $\bG=(V,E)$. 
As we mentioned in Section~\ref{sec:multim} when we introduced multimatroids, it is often useful to think of the elements of a skew class as representing different `states' that the skew class might play. This is particularly pertinent for the $3$-matroid associated with a ribbon graph and underlies our choice of notation.
Let $U(\bG):=\{\dott e, \overline e, \widehat e: e\in E\}$, 
 $\Omega:=\{\omega_e:e \in E\}$ and $\omega_e:=\{\dott e, \overline e, \widehat e\}$ for each $e$ in $E$. 
For each $(X,Y,Z)$ in $\mathcal P_3(E)$, we let 
\[T(X,Y,Z):= \{   \dott{e} : e\in X  \}\cup   \{   \overline{e} : e\in Y  \} \cup \{   \widehat{e} : e\in Z  \}.\]
The next result enables us to choose the set of bases.

\begin{theorem}\label{thm:newrg3}
Let $\bG=(V,E)$ be a  ribbon graph. Then the set
\[  \mathcal{B}(\bG):= \{   T(X,Y,Z) :  (X,Y,Z)\in \mathcal{Q}(\bG) \} \]
forms the collection of bases of a tight 3-matroid $Z$ with carrier $(U(\bG) ,\Omega)$. 
\end{theorem}
\begin{proof} 
Given a ribbon graph $\bG$, we may form its ribbon-graphic delta-matroid $D(\bG)$ and then the $3$-matroid $Z_3(D(\bG))$ as described just before the theorem with $\mu_1(e)=\dott{e}$, $\mu_2(e)=\overline{e}$ and $\mu_3(e)=\widehat{e}$. We shall show that  $\mathcal{B}(\bG)=\mathcal B(Z_3(D(\bG)))$.

A set $T(X,Y,Z)$ is in $\mathcal{B}(\bG)$ if and only if $(X,Y,Z) \in \mathcal Q(\bG)$. For a triple $(X,Y,Z)\in {\mathcal P}_3(E)$,
this occurs if and only if  
$b(\bG^{\tau(Z)}|_{X\cup Z}) = k(\bG)$ 
which happens if and only if $X\cup Z$ is a feasible set of $D(\bG^{\tau (Z)})$. By the remarks before the theorem this is equivalent to $X\cup Z$ being feasible in $D(\bG)+Z$. Finally, this occurs if and only if $T(X,Y,Z) \in  \mathcal B(Z_3(D(\bG)))$.
\end{proof}
The above proof motivated a short proof of Theorem~\ref{thm:newrg3} not passing through delta-matroids which can be found in~\cite{me}.

We denote the $3$-matroid with basis set $\mathcal B(\bG)$ and carrier $(U(\bG),\Omega)$ by $Z(\bG)$.
As an example, it can be checked that when $\bG$ is the ribbon graph in Figure~\ref{newfex1}, then $Z(\bG)$ is the multimatroid from Example~\ref{example}.

\medskip

We need to recognise which edges of $\bG$ give rise to  singular skew classes of $Z(\bG)$. For this we need some additional ribbon graph terminology.
An edge $e$ of a ribbon graph $\bG$ is a \emph{bridge} if $k(\bG \ba e)>k(\bG)$. It is a \emph{loop} if it is incident to exactly one vertex. A loop $e$ is \emph{orientable} if $e$ together with its incident vertex forms an annulus, and is \emph{nonorientable} if it forms a m\"obius band. A loop $e$ is said to be \emph{interlaced} with a cycle $C$ if when travelling around the boundary of the vertex incident to $e$ we see edges in the cyclic order $e \, c_1\, e \,c_2$ where $c_1$ and $c_2$ are edges of $C$, these edges will coincide in the case that $C$ is a loop.
A loop is \emph{trivial} if it is not interlaced with any cycle. 
Observe that a bridge necessarily intersects one boundary component of $\bG$, and a trivial orientable loop must meet two boundary components.

\begin{proposition}\label{prop:newsing}
Let $\bG=(V,E)$ be a ribbon graph.  
Then a skew class $\omega_e=\{\dott{e},\bar{e},\hat{e}\}$ of $Z(\bG)$ is singular if and only if $e$ is either a 
bridge or a trivial loop in $\bG$.

Furthermore, if $e$ is a bridge then $\bar{e}$ is a singular element, if $e$ is a trivial orientable loop then $\dott{e}$ is a singular element, and if $e$ is a trivial nonorientable loop then $\hat{e}$ is a singular element. 
\end{proposition}
\begin{proof}
We first prove the second part and as a consequence one direction of the first part.

Suppose that $(X,Y,Z)\in\mathcal{Q}(\bG)$. 
Consider first the case where $e$ is a trivial orientable loop. If $e\in X$, then $e$ is also a trivial orientable loop in $\bG^{\tau(Z)}\ba Y$ and because 
a trivial orientable loop necessarily meets two distinct boundary components it follows that
$b(\bG^{\tau(Z)}\ba Y) >k(\bG)$. Thus $e\notin X$ and 
$\dott{e}$ is not in any basis of $Z(\bG)$.
Next consider the case where $e$ is a trivial nonorientable loop. 
If $e\in Z$, then $e$ is  a trivial orientable loop in $\bG^{\tau(Z)}\ba Y$ and by the same argument as in the previous case it follows that
$b(\bG^{\tau(Z)}\ba Y) >k(\bG)$. Thus $e\notin Z$ and 
$\widehat{e}$ is not in any basis of $Z(\bG)$.

Finally consider the case where $e$ is a bridge of $\bG$. If $e\in Y$ then $b(\bG^{\tau(Z)}\ba Y) >k(\bG)$, so $e\notin Y$ and $\bar{e}$ is not in any basis of $Z(\bG)$.
Thus if $e$ is  a 
bridge or trivial loop in $\bG$ then the skew class $\omega_e=\{\dott{e},\bar{e},\hat{e}\}$ is singular.

\smallskip

It remains to prove that if $e$ is neither a trivial loop nor a bridge, then the skew class $\omega_e$ is non-singular in $Z(\mathbb G)$.

If $e$ is a non-loop, non-bridge edge, then let $F_1$ be the edge set of a maximal spanning forest of $\bG$ that contains $e$, and let $F_2$ be the edge set of a maximal spanning forest of $\bG$ that does not contain $e$. (Both $F_1$ and $F_2$ exist.) 
Then clearly $(F_1,E-F_1,\emptyset), (\emptyset,E-F_1,F_1), (F_2,E-F_2,\emptyset) \in \mathcal{Q}(\bG)$. It follows that $\dott{e}\in T(F_1,E-F_1,\emptyset)$, $\bar{e}\in T(F_2,E-F_2,\emptyset)$, and $\hat{e}\in T(\emptyset,E-F_1,F_1)$, so each of $\dott{e}$, $\bar{e}$, and $\hat{e}$ appears in some basis of $Z(\bG)$.

Next suppose $e$ is a non-trivial loop.
Let $C$ be the edge set of a cycle of $\bG$ which is interlaced with $e$, and let $f$ be an edge in $C$. Now let $F$ be the edge set of a maximal spanning forest of $\bG$ containing neither $e$ nor $f$ but containing the other edges of $C$. (Such a spanning forest must exist.)
If $e$ is an orientable loop then clearly $(F,E-F,\emptyset ), (F\cup\{f,e\},E-(F\cup\{f,e\}),\emptyset ), (F,E-(F\cup\{e\}), \{e\}) \in \mathcal{Q}(\bG)$, showing, respectively, that each of 
$\bar{e}$, $\dott{e}$, $\hat{e}$ is in some basis of $Z(\bG)$.
Finally if $e$ is a nonorientable loop then 
$(F,E-F,\emptyset )$,
$(F\cup\{e\},E-(F\cup\{e\}),\emptyset )$, 
$(F\cup\{f\},E-(F\cup\{e,f\}), \{e\}) \in \mathcal{Q}(\bG)$, giving that, respectively, each of 
$\bar{e}$, $\dott{e}$, $\hat{e}$ is in some basis of $Z(\bG)$.
This completes the proof of the proposition.
\end{proof}

\medskip

Above we described how to delete an edge $e$ of a ribbon graph $\bG$, resulting in a ribbon graph $\bG\ba e$. It is also possible to contract an edge $e$ of $\bG$, resulting in a ribbon graph denoted $\bG \con e$. We only use a few properties of contraction here: that every edge in a ribbon graph may be contracted,   $\bG \con e$ has one fewer edge than $\bG$, and that $\bG$ and $\bG \con e$ have identical boundary components.
A full appreciation of the definition of contraction is not necessary here, but for completeness we include one. If $u$ and $v$ are vertices incident to $e$ (it is possible that $u=v$ here), consider the boundary component(s) of $e\cup\{u,v\}$ as curves on $\bG$. For each resulting curve, attach a disc (which will form a vertex of $\bG\con e$) by identifying its boundary component with the curve. Delete $e$, $u$ and $v$ from the resulting complex, to get the ribbon graph $\bG\con e$.

\begin{lemma}\label{lem:newdcdc}
Let $\bG$ be a ribbon graph, $e$ be an edge of $\bG$ and $\omega_e=\{\dott{e},\bar{e},\hat{e}\}$ be its skew class. Then
\[  Z(\bG\con e) =  Z(\bG)|_{\dott{e}}, \quad
 Z(\bG\ba e) =  Z(\bG)|_{\bar{e}}, 
 \quad 
  Z(\bG^{\tau(e)} \con e ) =  Z(\bG)|_{\hat{e}}, 
\]
and if $\omega_e$ is  singular then
\[  Z(\bG)\ba \omega_e = Z(\bG \con e  )  = Z(\bG \ba e  ) =  Z(\bG^{\tau(e)} \con e  ).   \]
\end{lemma}
\begin{proof}
Let $e$ be an edge of a ribbon graph $\bG$.
If $e$ is not a bridge, then it is clear that $(X,Y,Z)\in \mathcal{Q}(\bG\ba e)$ if and only if $(X,Y\cup \{e\},Z)\in \mathcal{Q}(\bG)$, from which it easily follows that  $Z(\bG\ba e) =  Z(\bG)|_{\bar{e}}$. 
Next, if $e$ is not a trivial orientable loop, since $\bG$ and $\bG\con e$ have the same boundary components, we have 
$(X,Y,Z)\in \mathcal{Q}(\bG\con e)$ if and only if  $(X\cup \{e\},Y,Z)\in \mathcal{Q}(\bG)$, from which it follows that  $Z(\bG\con e) =  Z(\bG)|_{\dot{e}}$.
 Finally, again using the fact that contraction preserves boundary components, it is readily seen that if $e$ is not a trivial nonorientable loop, then
 $(X,Y,Z)\in \mathcal{Q}(\bG^{\tau(e)}\con e)$ if and only if $(X,Y,Z\cup\{e\})\in \mathcal{Q}(\bG)$, and it follows that $Z(\bG^{\tau(e)} \con e ) =  Z(\bG)|_{\hat{e}}$. 

To establish the outstanding cases in the first part of the lemma and to prove the second part, we may suppose that $e$ is either a bridge or a trivial loop. Equivalently, by Proposition~\ref{prop:newsing}, $\omega_e$ is singular in $Z(\mathbb G)$. We recall from Proposition~\ref{prop:singminorsanddelclass} that if $\omega$ is a singular skew class of a multimatroid $Z$, then for all $x$ in $\omega$, we have $Z|x = Z\ba \omega$.

If $e$ is a bridge, then 
$(X,Y,Z)\in \mathcal{Q}(\bG\ba e)$ if and only if $(X\cup \{e\},Y,Z)\in \mathcal{Q}(\bG)$. Thus $Z(\bG\ba e) =  Z(\bG)|_{\dott{e}}$, and as $\omega_e$ is singular, 
$Z(\bG)|_{\dott{e}} =  Z(\bG)|_{\bar{e}}$, giving 
$Z(\bG\ba e) = Z(\bG)|_{\bar{e}}$. Similarly,
if $e$ is a trivial orientable loop, then $(X,Y,Z)\in \mathcal{Q}(\bG\con e)$ if and only if $(X,Y\cup \{e\},Z)\in \mathcal{Q}(\bG)$. Thus 
$Z(\bG\con e) =  Z(\bG)|_{\bar{e}} =  Z(\bG)|_{\dott{e}}$. 
Finally, if $e$ is a trivial nonorientable loop then $(X,Y,Z)\in \mathcal{Q}(\bG^{\tau(e)}\con e)$ if and only if $(X,Y\cup \{e\},Z)\in \mathcal{Q}(\bG)$. Thus 
$Z(\bG^{\tau(e)}\con e) =  Z(\bG)|_{\bar{e}} =  Z(\bG)|_{\widehat{e}}$. 

Combining this with the first part of the proof completes the first part of the lemma. The second part then follows because 
$\omega_e$ is singular, so $Z(\bG)\ba \omega_e = Z|x$ for all $x$ in $\omega_e$.
\end{proof}

As shown in~\cite{MR2869185}, the topological transition polynomial  can be defined by the recursion relation (which holds for every edge $e$)
\begin{equation}\label{newdctopt}
Q(\bG)=\alpha_e \, Q(\bG\con e)+\beta_e \,Q( \bG\ba e)+\gamma_e\, Q(\bG^{\tau(e)} \con e)\end{equation} together with the boundary condition $Q(\bG)= t^{|V|}$ when $\bG=(V,\emptyset)$ is an edgeless ribbon graph. 
Here we have written  $Q(\bG)$  for $Q(\bG; (\boldsymbol\alpha, \boldsymbol\beta, \boldsymbol\gamma) , t)$ and also do this for $Q(\bG\con e)$, etc.\ in the obvious way.
A routine rewriting of Equation~\eqref{newdctopt} gives the following recursive definition of $Q(\bG)$.
\begin{equation}\label{newdctopt2}
Q(\bG)=
\begin{cases}
   t^{|V|} & \text{if $\bG=(V,\emptyset)$ is edgeless,}
   \\
 (\alpha_e+t\beta_e+\gamma_e) \,Q(\bG\con e) & \text{if $e$ is a bridge,}\\
  (t\alpha_e+\beta_e+\gamma_e)\, Q(\bG\ba e) & \text{if $e$ is a triv.\ orient.\ loop,}\\
   (\alpha_e+\beta_e+t\gamma_e) \,Q(\bG^{\tau(e)}\con e) & \text{if $e$ is a triv.\ nonorient.\ loop,}
   \\
   \alpha_e \, Q(\bG\con e)+\beta_e \,Q( \bG\ba e)+\gamma_e\, Q(\bG^{\tau(e)} \con e) &\text{otherwise.}
\end{cases}
\end{equation}

We now relate $Q(\bG)$ and $Q(Z(\bG))$.

\begin{theorem}\label{thm:new2trans}
let $\bG$ be a  ribbon graph, then 
\[ Q(\bG; (\boldsymbol\alpha, \boldsymbol\beta, \boldsymbol\gamma) , t)  = t^{k(\bG)}\cdot  Q(Z(\bG);\vect x, t)  ,\]
when $\alpha_e=\vect x_{\dott{e}}$, $\beta_e=\vect x_{\bar{e}}$, and $\gamma_e=\vect x_{\hat{e}}$ for each edge $e$ of $\bG$.
\end{theorem}
\begin{proof}
First note that when $\bG$ is edgeless $Q(Z(\bG);\mathbf{x},t)=1$. 
Next  suppose  $e$ is an edge of $\bG$ and $\omega=\{\dott{e},\bar{e},\hat{e}\}$ is its corresponding skew class in $Z(\bG)$.

If $e$ is not a trivial loop or bridge, then by Proposition~\ref{prop:newsing} the skew class $\omega$ is non-singular. Then applying Proposition~\ref{prop:delcon} and Lemma~\ref{lem:newdcdc} gives  
\begin{align*}
Q(Z(\bG);\mathbf{x},t) &= x_{\dott{e}} Q(Z(\bG)|\dott{e};\mathbf{x},t) +x_{\bar{e}} Q(Z(\bG)|\bar{e};\mathbf{x},t) +x_{\hat{e}} Q(Z(\bG)|\hat{e};\mathbf{x},t)
\\
&= 
x_{\dott{e}} Q(Z(\bG\con e);\mathbf{x},t) +x_{\bar{e}} Q(Z(\bG\ba e);\mathbf{x},t) +x_{\hat{e}} Q(Z(\bG^{\tau(e)} \con e);\mathbf{x},t).
\end{align*}
Next if $e$ is a bridge, then by Proposition~\ref{prop:newsing} the element $\bar{e}$ is singular. 
By Proposition~\ref{prop:delcon} and Lemma~\ref{lem:newdcdc} we have 
\[Q(Z(\bG);\mathbf{x},t) 
= 
(tx_{\bar{e}} + x_{\dott{e}} +x_{\hat{e}})Q(Z(\bG)\ba \omega;\mathbf{x},t)
= 
(tx_{\bar{e}} + x_{\dott{e}} +x_{\hat{e}})Q(Z(\bG \con e);\mathbf{x},t)
\]
Similarly, when $e$ is a trivial orientable loop, 
\[Q(Z(\bG);\mathbf{x},t) 
= 
(tx_{\dott{e}} + x_{\bar{e}} +x_{\hat{e}})Q(Z(\bG \ba e);\mathbf{x},t),
\]
and when it is a trivial nonorientable loop, 
\[Q(Z(\bG);\mathbf{x},t) 
= 
(tx_{\hat{e}} + x_{\dott{e}} +x_{\bar{e}})Q(Z(\bG^{\tau(e)}\con e);\mathbf{x},t).
\]
By comparing the above identities for $Q(Z(\bG);\mathbf{x},t)$ with those in  Equation~\eqref{newdctopt2} it is readily seen that 
$t^{-k(\bG)}\cdot Q(\bG; (\boldsymbol\alpha, \boldsymbol\beta, \boldsymbol\gamma) , t) $ and  $Q(Z(\bG);\mathbf{x},t)$ define the same ribbon graph invariants when   $\alpha_e=\vect x_{\dott{e}}$, $\beta_e=\vect x_{\bar{e}}$, and $\gamma_e=\vect x_{\hat{e}}$.
\end{proof}

\begin{remark}
Unlike the relationship between the Tutte polynomials of a graph and its cycle matroid, the relationship between
$Q(\bG)$ and $Q(Z(\bG))$ involves a prefactor. To remove this necessity one might define $Q(\bG)$ differently, by using Equation~\eqref{newdctopt2} but 
setting its value to one on every edgeless ribbon graph, but then Equation~\eqref{newdctopt} would need to be modified to deal with the cases where $e$ is a bridge, trivial orientable loop or trivial nonorientable loop. For example, when $e$ is a bridge, $\beta_e$ would have to be replaced by $t\beta_e$ to align with Equation~\eqref{newdctopt2}.
\end{remark}
 
The boundary of a ribbon graph is a collection of closed curves.
Suppose that $\bG=(V,E)$ is a ribbon graph and let $e$ be an edge of $\bG$. 
Given $(X,Y,Z)$ in ${\mathcal P}_3(E)$ each edge $e$ of $\bG$
induces two arcs $c$ and $c'$ on the boundary of $\bG^{\tau(Z)}\ba Y$. 
If $e\in X\cup Z$, these are the two arcs on the boundary of $\bG^{\tau(Z)}\ba Y$ intersecting $e$; for $e\in Y$ the arcs $c$ and $c'$ are the arcs that correspond to the intersection of $e$ and its incident vertices (or vertex) in $\bG$. Moving $e$ from one of $X$, $Y$ or $Z$ to another has the effect of deleting the arcs $c$ and $c'$ to give two open curves and closing them with two new arcs joining the ends of $c$ and $c'$. This may change the number of closed curves by at most one. 

\begin{lemma}\label{lem:activeinterlace}
    Let $\bG$ be a ribbon graph, let $(X,Y,Z)$ be in $\mathcal Q(\bG)$ and let $B:=T(X,Y,Z)$. Then distinct edges $e$ and $f$ of $\bG$ are $(X,Y,Z)$-interlaced if and only if $C(B,\omega_e)$ meets $\omega_f$. 

    Thus an edge $e$ of $\bG$ is active with respect to $(X,Y,Z)$ if and only if $\omega_e$ is active with respect to $B$.
\end{lemma}

\begin{proof}
The boundary of $\bG^{\tau(Z)}\ba Y$ has $k(\bG)$ components.  
Let $c$ and $c'$ be the two arcs on the boundary of $\bG^{\tau(Z)}\ba Y$ corresponding to $e$. Then the tightness of $Z_3(\bG)$ implies first that $C(B,\omega_e)$ exists and second that (exactly)
one of the two ways of moving $e$ from $X$, $Y$ or $Z$ to another yields a partition $(X',Y',Z')$ with $b(\bG^{\tau(Z')}\ba Y')=k(\bG)+1$. (One can also see this by considering the ways in which the ends of the arcs $c$ and $c'$ may be joined by a pair of arcs.)   

The boundary components of $\bG^{\tau(Z)} \ba Y$ and $\bG^{\tau(Z')} \ba Y'$ are identical except that
$\bG^{\tau(Z)} \ba Y$ has a boundary component $\partial$ that is not in $\bG^{\tau(Z')} \ba Y'$, and $\bG^{\tau(Z')} \ba Y'$ has two boundary components $\partial'$ and $\partial''$ that are not in $\bG^{\tau(Z)} \ba Y$.

We shall say that the edge $f$ of $\bG$ \emph{meets} $\partial'$ (respectively, $\partial''$)
if in $\bG^{\tau(Z')} \ba Y'$ either $f\in X'\cup Z'$ and
$f$ intersects $\partial'$ (respectively, $\partial''$) or $f\in Y'$
and $\partial'$ (respectively, $\partial''$) meets an arc on a vertex of $\bG^{\tau(Z')} \ba Y'$ that was an end of $f$. 

The edge $f$ is interlaced with $e$ if and only if $f$ meets both $\partial'$ and $\partial''$. If $f$ does not meet both $\partial'$ and $\partial''$, then no way of moving $f$ from one of $X'$, $Y'$ or $Z'$ to another yields a partition $(X'',Y'',Z'')$ with $b(\bG^{\tau(Z'')}\ba Y'')=k(\bG)$ as either $\partial'$ or $\partial''$ is unaffected such a move. 
If $f$ does meet both $\partial'$ and $\partial''$, then there is a way (in fact both possibilities work) of moving $f$ from one of $X'$, $Y'$ or $Z'$ to another to yield a partition $(X'',Y'',Z'')$ with $b(\bG^{\tau(Z'')}\ba Y'')=k(\bG)$. 

Thus we see that $f$ is interlaced with $e$ if and only if there is an element $x$ of $\omega_f$ such that $B\bigtriangleup \{\underline{B_{\omega_e}}, B_{\omega_e}, B_{\omega_f}, x\} \in \mathcal B(Z_3)$. It follows from Lemma~\ref{lem:blah} that $f$ is interlaced with $e$ if and only if $C(B,\omega_e)$ meets $\omega_f$.

The final part of the lemma follows immediately.
\end{proof}

Our final lemma before proving Theorem~\ref{newtt1} enables us to identify the elements $\underline B_{\omega_e}$. Given $(X,Y,Z)$ in $\mathcal Q(\bG)$, with say $e\in X$, then tightness of $Z(\bG)$ implies that precisely one of $(X-\{e\},Y\cup\{e\},Z)$ and $(X-\{e\},Y,Z\cup\{e\})$ is in $\mathcal Q(\bG)$, but we need to know which one.

\begin{lemma}\label{lem:whichisthebadguy}
    Let $\bG$ be a ribbon graph with edge $e$ and let $(X,Y,Z)$ be in $\mathcal Q(\bG)$. 
\begin{enumerate}
    \item If $e\in X$, then $(X-\{e\},Y\cup\{e\},Z) \notin \mathcal Q(\bG)$ if and only if $e$ is $(X,Y,Z)$-orientable. 
    \item If $e\in Y$, then $(X\cup\{e\},Y-\{e\},Z) \notin \mathcal Q(\bG)$ if and only if $e$ is $(X,Y,Z)$-orientable.
  
    \item  If   $e\in Z$  
    then $(X,Y\cup\{e\},Z-\{e\}) \notin \mathcal Q(\bG)$ if and only if $e$ is $(X,Y,Z)$-orientable.
\end{enumerate}    
\end{lemma}    
\begin{proof}
In each case the result follows by considering the arcs $c$ and $c'$ intersecting $e$ on the closed curve which is the boundary component of $\bG^{\tau(Z)}\ba Y$  containing $e$ or its ends, and the effect of deleting $c$ and $c'$ and reconnecting their ends. 
\end{proof}
 
\begin{proof}[Proof of Theorem~\ref{newtt1}]
Suppose that $\bG=(V,E)$ is a ribbon graph  and $Z(\bG)$ is its 3-matroid. We shall recover Theorem~\ref{newtt1} by expressing  Theorem~\ref{thm:main1} directly in terms of $\bG$.

It follows from Theorems~\ref{thm:new2trans} and~\ref{thm:main1} that
\[ Q(\bG;(\mathbf{\alpha}, \mathbf{\beta},\mathbf{\gamma}) = t^{k(\bG)}
\sum_{B\in \mathcal{B}(Z(\bG))}
\prod_{\omega \in \inact_{\prec}(B)} \vect x_{B_{\omega}}\prod_{\omega \in \act_{\prec}(B)}\left(  \frac {t\vect x_{\underline B_{\omega}}} {2}+\vect  x_{B_{\omega}} \right),\]
where $\mathbf x_{\dott{e}}=\alpha_e$,
$\mathbf x_{\overline{e}}=\beta_e$ and
$\mathbf x_{\dott{e}}=\gamma_e$.

It remains to interpret all the terms on the right side of the previous equation in terms of $\bG$.
Let $(X,Y,Z)$ be in $\mathcal{Q}(\bG)$, let $B:=T(X,Y,Z)$ and let $e$ be an edge of $\bG$. We consider the term in the product contributed by $e$ (or equivalently $\omega_e$) to the term  in the sum corresponding to $(X,Y,Z)$ (or equivalently $B$).

\begin{itemize}
    \item If $e$ is $(X,Y,Z)$-inactive then by Lemma~\ref{lem:activeinterlace} $\omega_e$ is inactive with respect to $B$ and the contribution of $e$ is $\alpha_e$, $\beta_e$ or $\gamma_e$, if $e\in X$, $e\in Y$ or $e\in Z$, respectively.
    \item If $e$ is active orientable, then by Lemma~\ref{lem:activeinterlace} $\omega_e$ is active with respect to $B$ and 
    it follows from Lemma~\ref{lem:whichisthebadguy} that
    $\underline B_{\omega_e}= \overline e$,
    $\underline B_{\omega_e}= \dott e$ and $\underline B_{\omega_e}= \overline e$, if $e\in X$, $e\in Y$ or $e\in Z$, respectively.

So the contribution of $e$ is $\frac{t\beta_e}2+ \alpha_e$,
$\frac{t\alpha_e}2+ \beta_e$ or
$\frac{t\beta_e}2+ \gamma_e$
if $e\in \AO(X)$, $e\in \AO(Y)$ or $e\in \AO(Z)$, respectively.

    \item If $e$ is active nonorientable, then by Lemma~\ref{lem:activeinterlace} $\omega_e$ is active with respect to $B$ and it follows from Lemma~\ref{lem:whichisthebadguy} that $\underline B_{\omega_e}= \widehat e$,
    $\underline B_{\omega_e}= \widehat e$ and $\underline B_{\omega_e}= \dott e$, if $e\in X$, $e\in Y$ or $e\in Z$, respectively.

So the contribution of $e$ is $\frac{t\gamma_e}2+ \alpha_e$,
$\frac{t\gamma_e}2+ \beta_e$ or
$\frac{t\alpha_e}2+ \gamma_e$
if $e\in \AN(X)$, $e\in \AN(Y)$ or $e\in \AN(Z)$,
respectively.
\end{itemize}
 \end{proof}

\bibliography{multimatroid-result}
\bibliographystyle{siamplain}

\end{document}